\documentclass[11pt]{amsart}

\usepackage{times,amsfonts,amsmath,amstext,amsbsy,amssymb,
amsopn,amsthm,upref,eucal}
\usepackage[T1]{fontenc}
\usepackage{yfonts}
\usepackage{mathrsfs}
\usepackage{color}
\usepackage{marginnote}

\newcommand \R {{ \mathbb R}}

\newcommand \Z {{ \mathbb Z}}
\newcommand \N {{ \mathbb N}}
\newcommand \T {{ \mathbb T}}

\newcommand \Gt {{ \Gamma\theta}}

\makeatletter
\newcommand{\opnorm}{\@ifstar\@opnorms\@opnorm}
\newcommand{\@opnorms}[1]{%
  \left|\mkern-1.5mu\left|\mkern-1.5mu\left|
   #1
  \right|\mkern-1.5mu\right|\mkern-1.5mu\right|
}
\newcommand{\@opnorm}[2][]{%
  \mathopen{#1|\mkern-1.5mu#1|\mkern-1.5mu#1|}
  #2
  \mathclose{#1|\mkern-1.5mu#1|\mkern-1.5mu#1|}
}
\makeatother

\newcommand \tF {{ \tilde F_{i, n}^\la}}

\def \L{{\Lambda}}
\def \la{{\lambda}}

\def \Ave {{\rm Ave}}

\def \dd {{\bf d}}

\def \d {{\delta}}

\def \half{{\frac{1}{2}}}
\def \exp{{\rm exp}}

\newcommand{\SO}{\operatorname{SO}}

\newcommand{\h}{\operatorname{\mathfrak h}}
\newcommand{\sH}{\mathsf H}
\newcommand{\e} {{\varepsilon }}

\newcommand{\bdeta} {{\boldsymbol \eta}}

\newcommand{\bdtau} {{\boldsymbol \tau}}

\newcommand{\vect} {{\rm Vect^{\infty}M }}

\newcommand{\diff} {{\rm Diff^{\infty}M }}

\newtheorem{theorem}{Theorem}[section]
\newtheorem {lemma} [theorem]{Lemma}
\newtheorem {proposition}[theorem]{Proposition}
\newtheorem{corollary}[theorem]{Corollary}
\newtheorem{remark}[theorem]{Remark}

\newtheorem{definition}[theorem]{Definition}

\title[Discrete Abelian actions on Heisenberg nilmanifolds]{Transversal local rigidity of discrete Abelian actions on Heisenberg nilmanifolds}
\author[Danijela Damjanovi\'c and 
James Tanis]{
Danijela Damjanovi\'c$^1$  and 
James Tanis}
\thanks{ $^1$ Based on research supported 
Swedish Research Council grant 2015-04644}

\address[Danijela Damjanovi\'c]{Department of Mathematics, Kungliga Tekniska H\"ogskolan, Lindstedtsv\"agen 25, 10044 Stockholm, Sweden}
\email{ddam@kth.se}

\address[James Tanis \footnote{
\textbf{Approved for Public Release; Distribution Unlimited. Case Number 18-2318.}
The second author's affiliation with The MITRE Corporation is provided for identification purposes only, and is not intended to convey or imply MITRE's concurrence with, or support for, the positions, opinions or viewpoints expressed by the authors.
} ]{The MITRE Corporation \\ McLean, VA 22102, USA}
\email{jhtanis@mitre.org}

\begin{document}

\begin{abstract}
In this paper we  prove a perturbative result for a class of $\Z^2$ actions on Heisenberg nilmanifolds, which have Diophantine properties. Along the way we prove cohomological rigidity and obtain a tame splitting for the cohomology with coefficients in smooth vector fields for such actions.  
\end{abstract}

\maketitle
\tableofcontents

\section{Introduction}


Starting with the seminal work of Katok and Spatzier on Anosov actions \cite{KS2}, smooth local classification of Abelian actions with hyperbolic features has deserved a lot of attention. Hyperbolicity implies existence of invariant geometric structures whose properties are exploited in obtaining very strong local classification results \cite{DK_PH2}, \cite{VW}.  The main goal of local classification is completely understanding the dynamics of smooth actions which are small perturbations of the given action.  

For actions with no hyperbolicity, such as parabolic and elliptic actions, there are no convenient invariant geometric structures and the methods from the hyperbolic theory are not applicable. Also, for parabolic and elliptic actions the local classification results are weaker than for hyperbolic actions, and the  methods used are more analytical. For elliptic abelian actions the main feature allowing local classification has been the Diophantine property \cite{Moser, P} for torus translations, while the main strategy for proving local classification results has been the method of successive iterations labeled in the 60's by KAM method after Kolmogorov, Arnold and Moser who devised it for the purpose of showing persistence of Diophantine tori in Hamiltonian dynamics. The method has been more recently adapted to certain kind of parabolic \emph{continuous} time actions in \cite{DK_UNI}, and later used in  \cite{D_GH, W2}. This adapted method is described for general Lie group actions in \cite{D_IFT}.

In this paper we apply this adapted KAM method of successive iterations to a class of \emph{discrete} time Abelian actions which are parabolic, meaning that the derivative of the action has polynomial growth.  We describe a class of discrete Abelian actions on a (2n+1)-dimensional Heisenberg nilmanifold, which on the induced torus have certain Diophantine properties. For the purpose of this introduction we call these actions "Diophantine".  
We show that these Diophantine actions  belong to a  \emph{finite dimensional} $4g-1$- dimensional family of algebraic actions for which we prove  a {local classification result}. Namely we show that a small perturbation of the family around the Diophantine member contains a smooth conjugate of that Diophantine action. This implies that every perturbed family contains an element which is dynamically the same as the Diophantine action. This phenomenon  has been previously labelled \emph{transversal local rigidity } and has been {studied} for classes of  \emph{continuous time} actions \cite{DK_UNI, D_GH}. For \emph{discrete} abelian actions, we are not aware of any results in the literature where \emph{transversal} local rigidity is proved and where 
it does not  follow from a stronger local (or global) rigidity result for actions of $\mathbb Z$ or $\mathbb R$.  

The analytic method of obtaining  local classification results interprets the local conjugation problem as a non-linear operator, which after linearisation describes the \emph{cohomology} over the unperturbed actions. 
The {linearized} version of the local classification problem is precisely \emph{the first cohomology group} with coefficients in smooth vector fields. If the first cohomology is  finite dimensional and both first and second coboundary operators have inverses with sufficiently nice \emph{tame} norm estimates, then one can reasonably hope to employ the KAM iterative method. Tameness means that the $C^r$ norm of the solution can be bounded by the  $C^{r+\sigma}$ norm of the given data, where $r$ is arbitrarily large while $\sigma$  is a constant.  In short, the analytic method has two major ingredients:  a detailed analysis of the first cohomology and coboundary operators, and { an} application of the KAM iteration. Such detailed analysis of cohomology is usually hard to perform, and usually needs to use { the} full machinery of the representation theory, which is why results are often restricted to actions on manifolds of smaller dimension and simpler structure of representation spaces. This is the main reason that there is lack of local rigidity results for parabolic actions on higher step nilmanifolds. 

We remark that even when careful analysis of first cohomology is possible, the inverses of coboundary operators may lack tameness in which case KAM method may not work. Namely, in  \cite{DT} we carried out analysis of the first cohomology for the discrete parabolic homogeneous action on $SL(2, \mathbb R)\times SL(2, \mathbb R)/\Gamma$. However, the inverse of the second coboundary operator turned out not to be tame, in fact  \cite{TW1} proved there can be no tame inverse (see also Theorem~ 2.2 of \cite{TW2}). No local classification results have been obtained for this example.

In this paper we perform detailed analysis of  cohomology for a class of discrete time actions with Diophantine properties on $2n+1$\color{black}-dimensional Heisenberg nilmanifold{s}. It turns out that their cohomology is finite dimensional and we can obtain tame estimates for solutions of coboundary operators. 
Once we get complete cohomological information, we use the KAM method to prove transversal local rigidity. This is similar to  the proof of the main results in \cite{D_IFT} and \cite{DK_UNI}\,, except that in the case of  discrete actions we have somewhat more complicated (linear and non-linear) operators to work with.  As far as we know this is the first example of a \emph{discrete parabolic} (but not elliptic) abelian action for which some kind of local rigidity property holds. 

The analysis of first cohomology for the corresponding \emph{continuous time} group actions on Heisenberg nilmanifolds has been carried out in \cite{CF}. In the continuation of the work presented in this paper, we intend to address local classification of the $\mathbb R^k$ actions described in \cite{CF} as well as their discrete subactions. 

\subsection{Setting}\label{setting}
Let $n\ge 2$ be an integer.  The Heisenberg group over $\R^n$ 
is the set $\sH := \sH(n) = \R^n \times \R^n \times \R$, and it is equipped with 
the group multiplication  
\[
(\mathbf x\,, \boldsymbol \xi\,, t) \cdot (\mathbf x'\,, \boldsymbol \xi'\,, t') = 
(\mathbf x + \mathbf x'\,, \boldsymbol \xi + \boldsymbol \xi'\,, t + t' + \frac{1}{2} (\boldsymbol\xi \cdot \mathbf x' - \mathbf x \cdot \boldsymbol \xi'))\,. 
\]
Lie algebra of $\sH$ is the vector space $\R^n \times \R^n \times \R$, 
which is generated by the vector fields 
\[
(X_i)_{i = 1}^n\,, \quad (\Lambda_i)_{i = 1}^n\,, \quad Z 
\]
that satisfy the commutation relations 
\[
[X_i, X_j] = 0\,, \ \ \ \  [\Lambda_i, \Lambda_j] = 0\,, \ \ \ \  [X_i, \Lambda_j] = \delta_{i j} Z\,, \ \ \ \  \ \ \ \ i, j \in \{1,  2, \ldots, n\} \,.
\]
The set $\Gamma := \Z^{n} \times \Z^{n} \times \frac{1}{2} \Z \subset \sH$ is \textit{standard lattice} of $\sH$. The lattice is co-compact and the compact quotient manifold $M := \Gamma \backslash \sH$ is called the 
\textit{standard Heisenberg nilmanifold}.  

Even though our proofs are written for the case of the standard lattice $\Gamma$, this not a restriction, the results in fact automatically hold for general lattices of $\sH$ due to the complete description of all lattices in  $\sH$  and the corresponding representation of $\Gamma \backslash \sH$ by Tolimieri in \cite{Tol}.  

Let $L^2(M)$ be the space of complex-valued square-integrable functions on $M$.  
As in \cite{CF}, we define the Laplacian on $L^2(M)$ by 
\begin{equation}\label{eq:laplacian}
\triangle :=  - Z^2 - \sum_{i = 1}^n X_i^2 + \Lambda_i^2 \,.
\end{equation}
Then $\triangle$ is an essentially self-adjoint, non-negative operator, and $(I + \triangle)^s$ is defined by the spectral theorem for all $s > 0$.  
The space $W^s(M)$ is the Sobolev space of $s$-differentiable functions defined to be the maximal domain of $(I + \triangle)^s$, and it is equipped with the inner product  
\begin{equation}\label{eq:inner_prod}
\langle f, g \rangle_s := \langle (I + \triangle)^s f, g \rangle\,.
\end{equation}
The norm of a function $f \in W^s(M)$ is denoted $\Vert f \Vert_s$.  
Because $M$ is compact, we have 
\[
C^\infty(M) := \cap_{s\geq 0} W^s(M)\,.
\]

For $\mathbf m \in \Z^{2n}$, let  
\begin{equation}\label{eq:m1,m2}
\begin{aligned}
& \mathbf m := (m_1, m_2, \ldots, m_{2n})\,, \\ 
& \mathbf m_1 := (m_1, m_2, \ldots, m_{n})\,, \quad \mathbf m_2 := (m_{n+1}, m_{n+2}, \ldots, m_{2n})\,.
\end{aligned}
\end{equation}
Then let   
\[
\boldsymbol{\tau} := (\tau_1, \tau_2, \ldots, \tau_n, \mathbf 0)\,, \quad \boldsymbol{\eta} := (\mathbf 0, \eta_1, \eta_2, \ldots \eta_n)  
\]
be Diophantine over $\Z^{n}$ in $\R^n$ and satisfy 
\begin{equation}\label{eq:orthogonality:00}
\sum_{j = 1}^n \tau_j \eta_j  = 0\,.
\end{equation}
By Diophantine, we mean that 
there are constants $c := c_{\boldsymbol{\tau}, \boldsymbol{\eta}} > 0$ and $\gamma := \gamma_{\boldsymbol{\tau}, \boldsymbol{\eta}} > 0$ such that  
for any $\mathbf m \in \Z^{2n}$ and $p \in \Z$,  we have 
\begin{equation}\label{eq:diophantine_property}
\begin{aligned}
& \vert \boldsymbol{\tau} \cdot \mathbf m - p \vert > c \vert \mathbf m_1 \cdot \mathbf m_1 \vert^{-\gamma} & \text{ if } \mathbf m_1 \neq  0\,, \\ 
&  \vert \boldsymbol{\eta} \cdot \mathbf m - p \vert > c \vert \mathbf m_2 \cdot \mathbf m_2 \vert^{-\gamma} & \text{ if } \mathbf m_2 \neq  0\,.
\end{aligned}
\end{equation}
Next let 
\[
Y_{\boldsymbol{\tau}} := \sum_{i = 1}^n \tau_i X_{i} \,, \quad Y_{\boldsymbol{\eta}} := \sum_{i = 1}^n \eta_i \Lambda_i\,, 
\]
and notice that these vector fields commute because  
\begin{equation}\label{eq:commute_vfs}
[Y_{\boldsymbol{\tau}}\,,  Y_{\boldsymbol{\eta}}] = 0
\end{equation} 
is equivalent to \eqref{eq:orthogonality:00}.   

We consider the $\Z^2$ right-action on $M$ given by 
\begin{equation}\label{def:Z^2-action}
\rho (m_1, m_2) (x) := x \ \exp(m_1 Y_{\boldsymbol{\tau}} + m_2 Y_{\boldsymbol{\eta}}).
\end{equation}
Action $\rho$ induces a $\Z^2$ action on $L^2(M)$ (which we also dento by $\rho$), defined by: 
\[
\rho(m_1, m_2)( f) := f\circ \rho(m_1, m_2)\,.
\]
\subsection{Results on cohomological rigidity}

Let $\rho:\Z^k\rightarrow \rm Diff^\infty(M)$
be a smooth $\Z^k$ action on a compact manifold $M$.  Let $V$ be a $\rho$-module, by which we mean that there is a $\Z^k$ action on $V$, which we label by $\rho_*$.  Let $C^l(\Z^k,V)$ denote the space of  multilinear maps from $\Z^k\times \dots \times\Z^k$ to $V$. 

Then we have the cohomology sequence:
\begin{equation}~\label{cohomology sequence}
 C(\Z^k,V) \xrightarrow{\dd_1} C^1(\Z^k,V)\xrightarrow{\dd_2}C^2(\Z^k,V),
\end{equation}
where  the operators $\dd_1$ and $\dd_2$ are defined as follows:\\
For $H \in C(\Z^k,V)=V$ and $\beta \in C^1(\Z^k,V)$ define
\begin{equation}\label{cohom_vector}
\begin{array}{cc}
\dd_1H(g) = \rho_*(g)H-H \\ 
(\dd_2\beta)(g_1,g_2) = (\rho_*(g_2)\beta(g_1)-\beta(g_1)) - (\rho_*(g_1)\beta(g_2)-\beta(g_2)).
\end{array}
\end{equation}
The first cohomology $H_\rho^1(V)$ over the action $\rho$ with coefficients in module $V$ is defined to be $\rm Ker (\dd_2)/\rm Im (\dd_1)$. Elements of $\rm Ker (\dd_2)$ are called \emph{cocycles} over $\rho$ with coefficients in $V$, and elements of $\rm Im (\dd_1)$ are called \emph{couboundaries} over $\rho$ with coefficients in $V$. 

We consider here two situations: 

\noindent 1. $V=C^\infty(M)$ and $\rho_*(g) f= f\circ \rho(g)$ for any $g\in \mathbb Z^k$ and any $f\in  C^\infty(M)$, and 

\noindent 2. $V=\vect$ and $\rho_*(g)X= D\rho(g)X\circ \rho(g)^{-1}$ for any $g\in \mathbb Z^k$ and any $X\in \vect$.

We say that  $H_\rho^1(C^\infty(M))$ is \emph{constant} if up to a  modification by a constant cocycle, every cocycle is a coboundary.  This means that $H_\rho^1(C^\infty(M))$ is isomorphic to $\mathbb R^k$.

Now let $M$ be the homogeneous space $\Gamma\setminus G$ where $G$ a Lie group with Lie algebra $\frak g$ and $\Gamma$ a lattice in $G$. Let $\rho$ be a $\mathbb Z^k$ action on $M$ by right multiplication. Then $\rho$ induces action $\rho_*$ on $\frak g$ via the adjoint operator $\rm ad$. This action makes $\frak g$ into a module so one can consider the cohomology $H_\rho^1( \frak g)$, which is of course finite dimensional. If $H_\rho^1( \vect)=  H_\rho^1( \frak g)$ i.e. if  the cohomology with coefficients in vector fields is the same as the cohomology over $\rho$ with coefficients in \emph{constant} vector fields, then 
we say $H_\rho^1( \vect)$ is \emph{constant}. In particular, constant $H_\rho^1( \vect)$ is \emph{exceptionally small}: it is finite dimensional.

\begin{theorem}\label{main-cohomology}
For the action $\rho$ defined in Section~\ref{setting}, both $H_\rho^1(C^\infty(M))$ and $H_\rho^1( \vect)$ are constant. Moreover, in both cases, the operators $\dd_1$ and $\dd_2$ have tame inverses. Namely,  there exist positive constants $\sigma$ and $s_0$, and there exists left inverse $\dd^*_i$ of $\dd^*_i$, for $i=1,2$, such that for all $s\ge s_0$ there is a constant $C_s > 0$ such that $\|\dd^*_i \gamma_i\|_{s}\le C_s \|\gamma_i\|_{s+\sigma}$, where $\gamma_i$ is a cochain in $\rm Im (\dd_i)$. 
\end{theorem}

The above theorem is a consequence of the following two results which contain precise information on estimates for the norms of solutions to cohomological equations, which is essential for application of KAM method. 

We define the first coboundary operators associated to the generators of $\rho$. These are 
operators $L_{\boldsymbol{\tau}}$ and $L_{\boldsymbol{\eta}}$ on $L^2(M)$ given by  
\begin{equation}\label{eq:1st_coboundary}
\begin{aligned}
& L_{\boldsymbol{\tau}}f := f\circ \rho(1, 0)-f\,, \\ 
& L_{\boldsymbol{\eta}}f := f\circ  \rho(0, 1) -f\,.
\end{aligned}
\end{equation}

\begin{theorem}\label{main:transfer}
For any $s \geq 0$ and for any $\epsilon > 0$, there is a constant $C_{s, \epsilon} := C_{s, \epsilon, \boldsymbol{\tau}, \boldsymbol{\eta}} > 0$ such that for any $f, g \in C^\infty(M)$ of zero average with respect to the Haar measure, and that satisfy $L_{\boldsymbol{\tau}} g = L_{\boldsymbol{\eta}} f$, there is a solution $P \in C^{\infty}(M)$ such that 
\[
L_{\boldsymbol{\tau}} P = f \text{ and } L_{\boldsymbol{\eta}} P = g\,,
\]  
and 
\[
\| P \|_{s} \leq C_{s, \epsilon} (\| f \|_{s + \max\{2\gamma, n+1+\epsilon\}} + \Vert g \Vert_{s + 2\gamma})\,,
\] 
where $\gamma$ is the Diophantine exponent in \eqref{eq:diophantine_property}.
\end{theorem}

\begin{theorem}\label{main:split} 
For any $s \geq 0$ and for any $\epsilon > 0$, there is a constant $C_{s, \epsilon}:= C_{s, \epsilon, \boldsymbol{\tau}, \boldsymbol{\eta}} > 0$ such that for any $f, g, \phi \in C^\infty(M)$ of zero average that satisfy $L_{\boldsymbol{\eta}} f - L_{\boldsymbol{\tau}} g = \phi$, there exists a nonconstant function $P \in C^\infty(M)$ such that 
\[
\begin{aligned}
& \|g - L_{\boldsymbol{\eta}} P\|_s \leq C_{s, \epsilon}\| \phi\|_{s+ \sigma(n, \gamma, \epsilon)}\,, \\ 
& \|f - L_{\boldsymbol{\tau}} P\|_s \leq C_{s, \epsilon}  \| \phi\|_{s+\sigma(n, \gamma, \epsilon)} \,, \\ 
& \| P \|_s \leq C_{s, \epsilon} (\| f\|_{s+\sigma(n, \gamma, \epsilon)} + \Vert g \Vert_{s + \sigma(n, \gamma, \epsilon)}) \,, 
\end{aligned}
\]
where $\sigma(n, \gamma, \epsilon) := \max\{2\gamma,5n/2+1+\epsilon\}$.  
\end{theorem}

\begin{remark}
Results of this section can be viewed as the first step of obtaining discrete counterpart of the results of Cosentino and Flaminio on Lie group actions on Heisenberg nilmanifolds \cite{CF}. {An additional difficulty in the discrete case is that the space of obstructions to solutions of the cohomological equation is infinite dimensional in each irreducible, infinite dimensional representation.} 
We trust that the following general result holds: for actions of Lie groups $P$ considered in \cite{CF}, every non-degenerate lattice subaction of $P$ satisfies the statement of Theorem \ref{main-cohomology}. 
\end{remark}

\begin{remark} 
The Diophantine constant{s} in \eqref{eq:diophantine_property} could have different values for $\boldsymbol{\tau}$ and for $ \boldsymbol{\eta}$. It would not effect results, only the values of the constants in the estimates. For simplicity we used the same $\gamma$ throughout.
\end{remark}

\begin{remark}
We note that for {a} typical element of the action $\rho$, the first cohomology is infinite dimensional as a consequence of the results of  Flaminio and Forni in \cite{FF}. The results in \cite{FF} hold for nilmanifolds of any step, and it is an interesting open problem to construct $\R^k$ and $\Z^k$ homogeneous actions satisfying Theorem \ref{main-cohomology} on nilmanifolds of step greater than 2.   
\end{remark}

\subsection{Transversal local rigidity result}

 Let $\rho$ be  a smooth action of a discrete group $G$  by diffeomorphisms  of a smooth compact manifold $M$. Suppose that there exists a \emph{finite dimensional family} $\{\rho^\la\}_{\la\in \mathbb R^d}$ of smooth $G$ actions on $M$ such that $\rho^0=\rho$, and the family is  $C^1$ transversally i.e. it is $C^1$ in the parameter $\la$.

 Action $\rho$ is \emph{transversally locally rigid with respect to the family} $\{\rho^\la\}$ if every sufficiently small perturbation of the family $\rho^\la$ in a neighborhood of $\la=0$ intersects the smooth conjugacy class of $\rho$, where the smooth conjugacy class of $\rho$ consists of all actions $\{h\circ \rho\circ h^{-1}: \,\, h\in \rm Diff^\infty(M)\}$.

\begin{theorem}\label{conj:conjugacydis} Let $\rho$ be the $\mathbb Z^2$ action defined in \eqref{def:Z^2-action} 
where $\boldsymbol{\tau}$ and $\boldsymbol{\eta}$ are Diophantine as in \eqref{eq:diophantine_property}.  
Then  $\rho$ is transversally  locally rigid with respect to an explicit  $(4n-1)$-dimensional family of homogeneous  $\mathbb Z^2$ actions.
\end{theorem}

The explicit family of actions is defined in Section \ref{constantdis}.

\subsection{Structure of the paper}

The paper has two parts with analysis of different flavor. In Section~\ref{part1} we prove the cohomological results in Theorems~\ref{main:transfer} and \ref{main:split}. These results are further used in Section~\ref{cohVF} to prove the Proposition \ref{splittingVF}. All these results together imply directly Theorem \ref{main-cohomology}. The main analytic tool for the proof of cohomological results is representations theory on the Heisenberg nilmanifold. The calculation in finite dimensional representations is significantly simpler and is written in the appendix. The main calculation in infinite dimensional representation is done in Section~\ref{subsect:scro}. In the second part of the paper we apply cohomological results to prove Theorem \ref{conj:conjugacydis}. We describe the finite dimensiional family relative to which transversal rigidity holds, in Section~\ref{family} and we prove the main iterative step needed for the Theorem \ref{conj:conjugacydis} in Section~\ref{conv}.

\section{Proofs of Theorems~\ref{main:transfer} and \ref{main:split}}\label{part1}

\subsection{Representation Spaces}\label{sect:rep_spaces}

Let $L^2(M)$ be the Hilbert space of complex-valued square integrable functions with respect to the $\sH$-invariant volume form for $M$.  
By the Stone-von Neumann theorem, 
the space $L^2(M)$ decomposes into an orthogonal sum of 
irreducible, unitary representations that are unitarily equivalent certain 
one dimensional or infinite dimensional models that we describe at the top of Sections~\ref{subsect:finite} and \ref{subsect:scro}.  
Moreover, by irreducibility, Sobolev spaces $W^s(M)$ are also decomposable in the above sense, 
because vector fields in $\h$ split into irreducible, unitary 
representation spaces, and the infinitesimal representations of $\h$ 
extend to representations of the enveloping algebra.    
For this reason, we may prove our Sobolev estimates concerning coboundary operators (Theorems~\ref{main:transfer} and \ref{main:split}) in simpler, orthogonal components of $W^s(M)$, and then we glue the estimates together at the end (see \eqref{eq:main_pf:0}).  

\subsection{Finite dimensional representations}\label{subsect:finite}
The one dimensional representations are unitarily equivalent to characters $\rho_{\mathbf m}$ of $\R^{2n}$ in $L^2(\T^{2n})$, for $\mathbf m \in \Z^{2n}$, and are given by 
\begin{equation}\label{eq:finite_rep}
\rho_{\mathbf m}(\mathbf x, \boldsymbol \xi, t) f = e^{2\pi i \mathbf m \cdot (\mathbf x, \boldsymbol \xi)} f\,. 
\end{equation}
For each integer $1 \leq j \leq n$, the derived representations of $\rho_{\mathbf m}$ are 
\[
X_j = 2\pi i m_j\,, \ \ \ \ \ \ \ \ \ \Lambda_j = 2\pi i m_{n + j}\,, \ \ \ \ \ \ \ \ \ Z = 0 \,. 
\]
Write 
\[
\rho := \bigoplus_{\mathbf m \in \Z^4} \rho_{\mathbf m}\,.
\]
So given $f \in L^2(\T^{2n})$, we have the orthogonal decomposition 
\[
f(\mathbf x, \boldsymbol \xi) = \sum_{\mathbf m \in \Z^4} f_{\mathbf m}  e^{2\pi i \mathbf m \cdot (\mathbf x, \boldsymbol \xi)}\,,
\]
where $\triangle$ acts on irreducible, unitary representations of $L^2(\T^{2n})$ by  
\[
\rho(\triangle) = 4\pi^2  \mathbf{m} \cdot \mathbf{m}\,.
\]

For $s > 0$, the subspace of $s$-differentiable functions is 
$W^s(\T^{2n}) \subset L^2(\T^{2n})$, defined to be the 
maximal domain of the operator $(I + \rho(\triangle))^{s/2}$ 
on $L^2(\T^{2n})$ with inner product and norm given by 
\eqref{eq:inner_prod}.  In particular, 
\begin{equation}\label{eq:norm_sob_finite_dim}
\Vert f \Vert_s^2 = \sum_{\mathbf m \in \Z^{2n}} (1 + 4\pi^2 \mathbf{m} \cdot \mathbf{m})^s \vert f_\mathbf{n} \vert^2 \,.
\end{equation} 
We denote the space of smooth functions in $L^2(\T^{2n})$ by 
\[
W^\infty(\T^{2n}) := \cap_{s \geq 0} W^s(\T^{2n})\,.
\]
Furthermore, for every $s$, we have $W^s(\T^{2n}) = \mathbb C \langle 1 \rangle \oplus W_0^s(\T^{2n})$, where $W_0^s(\T^{2n})$ is the Sobolev space of $s$-differentiable, zero average functions on $\T^{2n}$.  So it follows that 
\[
W^\infty(\T^{2n}) = \mathbb C \langle 1\rangle \oplus W_0^\infty(\T^{2n})\,,
\]
where $W_0^\infty(\T^{2n}) = \cap_{s \geq 0} W_0^s(\T^{2n})$.  
\smallskip

The below two propositions establish Theorems~\ref{main:transfer} and \ref{main:split} in the case of finite dimensional representations.  
The proofs are straightforward and deferred to the appendix.  
\begin{proposition}\label{prop:finite-reps-cocycle}
There is a constant $C_{\boldsymbol{\tau}, \boldsymbol{\eta}} > 0$ 
such that for any zero average $f, g \in W_0^\infty(\T^{2n})$ that satisfy $L_{\boldsymbol{\tau}} g = L_{\boldsymbol{\eta}} f$, there is a solution $P \in W^{\infty}(\T^{2n})$ such that 
\[
L_{\boldsymbol{\tau}} P = f \text{ and } L_{\boldsymbol{\eta}} P = g\,,
\]  
and for any $s \geq 0$\,, 
\[
\| P \|_{s} \leq C_{\boldsymbol{\tau}, \boldsymbol{\eta}} (\| f \|_{s+2\gamma} + \Vert g \Vert_{s + 2\gamma})\,.
\] 
\end{proposition}

\begin{proposition}\label{prop:finite-reps}
There is a constant $C_{\boldsymbol \tau, \boldsymbol \eta} > 0$ such that for any $\phi \in W^\infty(\T^{2n})$ and any nonconstant zero average functions 
$f, g\in W_0^\infty(\T^{2n})$ that satisfy 
$L_{\boldsymbol{\eta}} f - L_{\boldsymbol{\tau}} g = \phi$ 
there is a nonconstant function $P \in W^\infty(\T^{2n})$ such that for any $s \geq 0$, 
\[
\begin{aligned}
& \|g - L_{\boldsymbol{\eta}} P\|_s \leq C_{\boldsymbol \tau, \boldsymbol \eta} \|\phi \|_{s+2\gamma}  \,, \\ 
& \|f - L_{\boldsymbol{\tau}} P\|_s \leq C_{\boldsymbol \tau, \boldsymbol \eta} \|\phi \|_{s+2\gamma} \,, \\ 
& \| P \|_s \leq C_{\boldsymbol \tau, \boldsymbol \eta} (\Vert f \Vert_{s + 2\gamma} + \| g \|_{s+2\gamma})\,.
\end{aligned}
\]
\end{proposition}

\subsection{Schr$\ddot{\text{o}}$dinger representations}\label{subsect:scro}

Next we consider the infinite dimensional representations.  
Any infinite dimensional representation is unitarily 
equivalent to a Schr$\ddot{\text{o}}$dinger representation 
of $\sH$ on $L^2(\R^n)$ with a parameter $h \in 2\pi \Z\setminus\{0\}$.   
When acting on the right, this is  
\begin{equation}\label{eq:schrodinger}
(\mu_{h}(\mathbf{x},\boldsymbol \xi, t) \phi)(y) 
= e^{-i h t + i \epsilon \vert h \vert^{1/2} \boldsymbol \xi\cdot y - \frac{1}{2} i h \boldsymbol \xi\cdot \mathbf{x}} \phi(y-\vert h \vert^{1/2} \mathbf{x})\,,
\end{equation}
where $\epsilon = \text{sign}(h) = \pm 1.$  
For integers $1 \leq j \leq n$, we have 
\[
\mu_h(X_j) = -\vert h \vert^{1/2} \frac{\partial}{\partial y_j}\,, \ \ \ \mu_h(\Lambda_j) =  i \epsilon \vert h \vert^{1/2} y_j\,, \ \ \  \mu_h(Z) = -i h\,.
\]

The derived representation extends to the enveloping algebra of the Lie algebra of $\sH$.  
Observe 
\[
\vert Z \vert = \vert h \vert\,, 
\]
and define the operator $\Box$ in the model $\mu_h$ to be  
\[
\begin{aligned} 
\mu_h(\Box) & := \vert \mu_h(Z) \vert -\sum_{i = 1}^n \mu_h(X_i^2) + \mu_h(\Lambda_i^2) \\ 
& = \vert h \vert (1 + \sum_{i = 1}^n y_i^2 -\frac{\partial^2}{\partial y_i^2})\,,  
\end{aligned} 
\]
which is homogeneous in $\vert h \vert$.  
Define $W^s(\mu_h, \R^n) \subset L^2(\R^n)$ to be Hilbert Sobolev space of $s$-differentiable functions that is the maximal domain of the operator $\mu_h(\Box)^s$ on $L^2(\R^n)$ with inner product 
\[
\langle \mu_h(\Box)^s f, g \rangle_{L^2(\R^n)} = \vert h \vert^s \langle (I + \sum_{i = 1}^n y_i^2 -\frac{\partial^2}{\partial y_i^2})^s f, g\rangle_{L^2(\R^n)}\,.   
\] 
Denote the Sobolev norm of this operator by 
\begin{equation}\label{eq:homogeneous_norm}
\opnorm{f}_s := \langle \mu_h(\Box)^s f, f \rangle_{L^2(\R^n)}\,.
\end{equation}
Clearly, the space of smooth functions in $L^2(\R^n)$ with respect to $\mu_h(\Box)$ is the Schwartz space 
\[
\mathscr{S}(\R^n) = \cap_{s \geq 0} W^s(\mu_h, \R^n)\,.
\]

Analogous to Lemma~3.15 of \cite{CF}, estimates of linear operators in the full Laplacian \eqref{eq:laplacian} follow from such estimates in the above homogeneous norm.  
\begin{lemma}\label{lemm:reduce_h1}
Let $T: \mathscr{S}(\R^n) \to \mathscr{S}(\R^n)$ be a linear map for the representation $\mu_h$ such that for every $s \geq 0$, there is a constant $C_s > 0$ and some $t \geq 0$ satisfying 
\[
\opnorm{T f}_s \leq C_{s} \opnorm{f}_{s + t}\,.
\]
Then for every $s \geq 0$, there is another constant $C_s > 0$ such that 
\[
\Vert T f \Vert_s \leq C_s \Vert f \Vert_{s + t}\,.
\]
\end{lemma}
\begin{proof}
First let $s \geq 0$ be an integer.  
Then 
\begin{align}
\Vert T f \Vert_s^2 & = \langle \Big(I -\mu_h(Z)^2 -\sum_{i = 1}^n \mu_h(X_i^2) + \mu_h(\Lambda_i^2\Big)^s T f, T f\rangle \notag \\ 
& = \langle \Big((1 + h^2) - \sum_{i = 1}^n X_i^2 + \Lambda_i^2\Big)^s T f, T f\rangle \notag \\ 
& = \sum_{k = 0}^s {s \choose k} (1 + h^2)^{s - k} \langle (-\sum_{i = 1}^n X_i^2 + \Lambda_i^2)^k T f, T f\rangle \notag \\  
& \leq \sum_{k = 0}^s {s \choose k} (1 + h^2)^{s - k} \langle \mu_h(\Box)^k T f, T f\rangle \notag \\ 
& \leq C_s \sum_{k = 0}^s {s \choose k} (1 + h^2)^{s +t - (k+t)} \langle \mu_h(\Box)^{k + t}  f, f\rangle\,. \label{eq:vert_norm_2}
\end{align}
Now because all terms are positive, 
\[
\begin{aligned}
\eqref{eq:vert_norm_2} & \leq C_s \sum_{k = 0}^{s + t} {s \choose k}(1 + h^2)^{s - k} \langle \mu_h(\Box)^{k}  f, f\rangle \\ 
& = C_s \langle (I + \vert \mu_h(Z) \vert + \mu_h(\triangle))^{s + t} f, f \rangle \\ 
& \leq 2 C_s \Vert f \Vert_{s + t}^2\,.
\end{aligned}
\]
The estimate for $s \geq 0$ follows by interpolation.  
\end{proof} 

We will use the above lemma to reduce our estimates to the case $h = 1$.  
Because the norm \eqref{eq:homogeneous_norm} is homogeneous in $h$, by rescaling by the factor $\vert h \vert^{s/2}$ from $\opnorm{f}_{s}$, we can restrict ourselves to the case $\vert h \vert = 1$, as in \cite{CF}.  
In what follows, we set $h = 1$, as the argument for $h = -1$ is analogous.  

Then to simplify notation, we write 
 \[
 X_j = -\frac{\partial}{\partial y_j}\,, \quad \Lambda_j = i y_j \,, \quad Z = -i\,, 
 \]
and we refer to the Schr$\ddot{\text{o}}$dinger representation on $L^2(\R^n)$ as 
 \[
 \mu := \mu_1\,.
 \]
For $s > 0$, we denote $W^s(\R^n) := W^s(\mu, \R^n)$.  
 
 It will be convenient to define the Sobolev space $W^s(\R^{n-1})$ 
 that is the maximal domain of the operator 
 $I + \sum_{i = 2}^{n} y_i^2 - \frac{\partial^2}{\partial y_i^2}$ on $L^2(\R^{n-1})$.  
 We use the same notation for the inner product, where in this setting  
 \[
 \langle f, g \rangle_s := \langle (I + \sum_{i = 2}^{n} y_i^2 - \frac{\partial^2}{\partial y_i^2})^{s} f, g\rangle_{L^2(\R^{n-1})}\,.
 \]
 The norm for $W^s(\R^{n-1})$ is denoted $\vert f \vert_s$.  

 \subsubsection{Change of variable}
 Define 
\[
\tau = \sqrt{\sum_{j = 1}^{n} \tau_j^2}\,.  
\]
Let $A = [\mathbf a_1, \mathbf a_2, \ldots, \mathbf a_n] \in O(n)$ be a $n\times n$ matrix 
with orthonormal rows $\mathbf a_i$ such that 
\[
\mathbf a_1 = \frac{1}{\tau} (\tau_1, \tau_2, \ldots, \tau_n)\,.
\]
Observe that $(\tau_j)$ and $(\eta_j)$ 
span a two dimensional subspace of $\R^n$, 
so we can choose $\mathbf a_2$ to be such that 
\[
(\eta_1, \eta_2, \ldots, \eta_n) \in \textrm{span}( \{\mathbf a_1, \mathbf a_2\})\,.
\]
Further choose the signs of the vectors $\mathbf a_j$, for $2 \leq j \leq n$, 
so that $A \in \SO(n)$.  
Then $A$ is the determinant one rotation of $\R^n$ such that 
\begin{equation}\label{eq:def_A}
\begin{aligned}
& A(\tau_1, \tau_2, \ldots, \tau_n) = (\tau, 0, \ldots, 0)\,. \\ 
& A(\eta_1, \eta_2, \ldots, \eta_n) = (\nu_1, \nu_2, 0, \ldots, 0)\,, 
\end{aligned}
\end{equation}
for some $(\nu_1, \nu_2) \in \R^2$.

For $y = (y_1, y_2, \ldots, y_n)$, 
define $z = (z_1, z_2,  \ldots, z_n)$ 
via matrix-vector multiplication by 
\[
z = A y\,.
\]
Therefore, 
\begin{equation}\label{eq:rotateY-f}
f(y) :=  f \circ A^{-1} (z)\,.
\end{equation}

Clearly, because $A$ is an orthogonal matrix, 
the operator 
$U_A: L^2(\R^n, dy) \to L^2(\R^n, dz)$ 
 given by 
\[
U_A f = f \circ A^{-1}
\]
is unitary.  
Let $\tilde \mu$ be the representation on $\sH$ such that for any $g \in \sH$, $\tilde \mu(g): L^2(\R^n, dz) \to L^2(\R^n, dz)$ is given by 
\[
\tilde \mu(g) := U_A \mu(g) U_A^{-1}\,.
\] 
So $\tilde \mu$ unitarily equivalent to $\mu$.  

Now we compute a basis for $\h$ in terms of the derived representations of $\tilde \mu$.  
For each $j \in \{1, 2, \ldots, n\}$, let $(x_{j, t}, \lambda_{j, t}, z_{j, t})_{t\in [-1,1]}$ be smooth curves in $\sH$ such that 
\[
X_j = \frac{d}{dt}\mu(x_{j, t}) |_{t = 0}\,, \quad \Lambda_j = \frac{d}{dt}\mu(\lambda_{j, t})|_{t = 0}\,, \quad Z_j = \frac{d}{dt}\mu(z_{j, t})|_{t = 0}\,.
\]  
Then set 
\[
\tilde X_j = \frac{d}{dt}\tilde \mu(x_{j, t}) |_{t = 0}\,, \quad \tilde \Lambda_j = \frac{d}{dt}\tilde \mu(\lambda_{j, t})|_{t = 0}\,, \quad \tilde Z_j = \frac{d}{dt}\tilde \mu(z_{j, t})|_{t = 0}\,.
\]

Let $A^{-1}$ be the matrix 
\[
A^{-1} = (b_{i j})
\]
for some coefficients $b_{i j}$.  A calculation shows that for $j \in \{1, 2, \ldots, n\}$, 
\[
\tilde X_j = -\sum_{k = 1}^n b_{j k} \frac{\partial}{\partial z_k} \,, \quad \tilde \Lambda_j = i \sum_{j = k}^{n} b_{j k} z_k \,, \quad  \tilde Z = -i\,.
\]
One can check that these operators satisfy the commutation relations 
\[
[\tilde X_i, \tilde X_j] = 0\,, \ \ \ \  [\tilde \Lambda_i, \tilde \Lambda_j] = 0\,, \ \ \ \  [\tilde X_i, \tilde \Lambda_j] = \delta_{i j} \tilde Z\,, 
\]
for $i, j \in \{1,  2, \ldots, n\}$.  
\begin{lemma}\label{lemm:sub-laplace}
We have 
\[
\tilde \mu(\Box) = I + \sum_{i = 1}^n z_i^2 - \frac{\partial^2}{\partial z_i^2}\,.
\]
\end{lemma}
\begin{proof} 
By definition, 
\begin{equation}\label{eq:sub:1}
\tilde \mu(\Box) = I + \sum_{i = 1}^n -\tilde X_i^2 - \tilde \Lambda_i^2\,. 
\end{equation}
Notice that 
\[
\begin{aligned}
\tilde X_i^2 & = \left(-\sum_{j = 1}^n b_{i j} \frac{\partial}{\partial z_j}\right)^2 \\ 
&  = \sum_{j, m = 1}^n b_{i j} b_{i m} \frac{\partial^2}{\partial z_j \partial z_m}\,.  
\end{aligned}
\] 
Because the columns of $A^{-1}$ are orthonormal, we get 
\[
\begin{aligned}
\sum_{i = 1}^n \tilde X_i^2 & = \sum_{i = 1}^n \sum_{j, m = 1}^n b_{i j} b_{i m} \frac{\partial^2}{\partial z_j \partial z_m}\,, \\
& = \sum_{j, m = 1}^n \sum_{i = 1}^n b_{i j} b_{i m} \frac{\partial^2}{\partial z_j \partial z_m} \\ 
& = \sum_{1 \leq j \neq m \leq n} \frac{\partial}{\partial z_j \partial z_m} \sum_{i = 1}^n b_{i j} b_{i m} + 
\sum_{j = 1}^n \frac{\partial^2}{\partial z_j^2} \sum_{i = 1}^n b_{i j}^2  \\ 
& = \sum_{j = 1}^n \frac{\partial^2}{\partial z_j^2}\,.  
\end{aligned}
\] 
Similarly, 
\[
\begin{aligned}
\sum_{i = 1}^n \tilde \Lambda_i^2 
& = - \sum_{j, m = 1}^n \sum_{i = 1}^n b_{i j} b_{i m} z_j z_m \\ 
& = -\sum_{1 \leq j \neq m \leq n} z_j z_m \sum_{i = 1}^n b_{i j} b_{i m} - 
\sum_{j = 1}^n z_j^2 \sum_{i = 1}^n b_{i j}^2  \\ 
& = - \sum_{j = 1}^n z_j^2\,.  
\end{aligned}
\] 

Hence, 
\[
\eqref{eq:sub:1} = I + \sum_{i = 1}^n z_i^2 - \frac{\partial^2}{\partial z_i^2}\,.
\]
\end{proof}

Finally, we compute the operator $\tilde\mu(\exp(Y_{\boldsymbol{\kappa}}))$, 
for $\boldsymbol\kappa \in \{\boldsymbol\tau, \boldsymbol\eta\}$.  
\begin{lemma}\label{lemm:basic_oper}
For any $f \in L^2(\R^n)$ and $z \in \R^n$, we have 
\[
\begin{aligned}
& \tilde\mu(\exp(Y_{\boldsymbol{\tau}})) f(z) = f(z - (\tau, 0, \ldots, 0))\,, \\ 
& \tilde\mu(\exp(Y_{\boldsymbol{\eta}})) f(z) = \exp(i\nu_2 z_2) f(z)\,,
\end{aligned}
\]
for some $\nu_2 \in \R^*$.  
\end{lemma}
\begin{proof} 
To help keep track of which coordinate system we are working in, note $U_A f = f\circ A^{-1}$, where $z = A y$.  So 
\[ 
U_A: L^2(\R^n, dy) \to L^2(\R^n, dz)\,, \quad U_A^{-1}: L^2(\R^n, dz) \to L^2(\R^n, dy)\,,  
\] 
and of course the Schr$\ddot{\text{o}}$dinger representation $\mu$ satisfies 
\[
\mu(g): L^2(\R^n, dy) \to L^2(\R^n, dy)\,, 
\]
for any $g \in \sH$.  
Then 
\begin{align}
\tilde \mu(\exp(Y_{\boldsymbol \tau})) f(z) & := U_A \mu(\exp(Y_{\boldsymbol \tau})) U_A^{-1} f(z) \notag \\ 
& = \mu(\exp(Y_{\boldsymbol \tau}))U_A^{-1} f(A^{-1} z) \notag \\ 
& =  \mu(\exp(Y_{\boldsymbol \tau})) U_A^{-1} f(y) \notag \\ 
& = U_A^{-1} f(y_1 - \tau_1, \ldots , y_n - \tau_n) \notag \\ 
& = f(A(y_1 - \tau_1, \ldots , y_n - \tau_n)) \notag \\ 
& = f(Ay - A(\tau_1, \ldots , \tau_n)) \notag \\ 
& = f(z - (\tau, 0, \ldots, 0))\,. \notag
\end{align}

Next, recall that $\boldsymbol \eta = (\mathbf 0, \eta_1, \eta_2, \ldots \eta_n) \in \R^{2n}$, and define $\underline{\boldsymbol\eta} := (\eta_1, \eta_2, \ldots, \eta_n)$.  
Then $\mu(\exp(Y_{\boldsymbol{\eta}}))$ is the multiplication operator 
\[
\mu(\exp(Y_{\boldsymbol{\eta}})) f (y)= e^{i \underline{\boldsymbol{\eta}} \cdot y} \cdot f(y)\,.
\]
So 
\begin{align}
\tilde \mu(Y_{\boldsymbol{\eta}}) f (z) & := U_A \mu(\exp(Y_{\boldsymbol{\eta}})) U_A^{-1} f(z) \notag \\ 
& = \mu(\exp(Y_{\boldsymbol{\eta}})) U_A^{-1} f (A^{-1} z) \notag \\ 
& = \mu(\exp(Y_{\boldsymbol{\eta}})) U_A^{-1} f (y) \notag \\ 
& = e^{i \underline{\boldsymbol{\eta}} \cdot y} U_A^{-1} f(y) \notag \\ 
& = e^{i \underline{\boldsymbol{\eta}} \cdot y} f(Ay) \notag \\ 
& = e^{i \underline{\boldsymbol{\eta}} \cdot A^{-1} z} f(z) \notag \\ 
& = e^{i A \underline{\boldsymbol{\eta}} \cdot z} f(z) \,. \label{eq:tilde_y_eta}
\end{align}

Now recall from \eqref{eq:def_A} that 
$A\underline{\boldsymbol{\eta}} = (\nu_1, \nu_2, 0, \ldots, 0)\,,$ 
for some $(\nu_1, \nu_2) \in \R^2$.  
so 
\begin{equation}\label{eq:tilde_y_eta:2}
\eqref{eq:tilde_y_eta} = \exp(i (\nu_1 z_1 + \nu_2 z_2)) f(z)\,.  
\end{equation}
Furthermore, observe that the assumption  
$[Y_{\boldsymbol{\tau}}, Y_{\boldsymbol{\eta}}] = 0$ 
from \eqref{eq:commute_vfs} 
is equivalent to the condition 
\[
\sum_{j = 1}^n \tau_j \eta_j  = 0\,.
\]
We also have $A^{-1} (\tau, 0, \ldots 0) = (\tau_j)$, 
where $A^{-1} = (b_{i j})$.  
Then for all $1 \leq j \leq n$, 
\[
b_{j 1} = \frac{\tau_j}{\tau}\,.
\]
Because $A \in \SO(n)$, we have 
\[
a_{1 j} = b_{j 1} = \frac{\tau_j}{\tau}\,.
\]
Hence, 
  \begin{align} 
 \nu_1 & = (A \boldsymbol \eta)_1  = \sum_{j = 1}^n a_{1 j} \eta_j = \frac{1}{\tau}\sum_{j = 1}^n \tau_j \eta_j = 0\,. \label{eq:nu1=0}
 \end{align}
 
Because $A$ is a rotation and $\nu_1 = 0$, 
we get that $\vert \nu_2\vert = \vert \boldsymbol\eta\vert > 0$.  
Finally, because $A$ is a real matrix and $\boldsymbol\eta \in \R^n$, it follows that $\nu_2 \in \R^*$.  
The lemma now follows from \eqref{eq:tilde_y_eta:2} and \eqref{eq:nu1=0}.  
\end{proof}

For $\boldsymbol\kappa \in \{\boldsymbol\tau, \boldsymbol\eta\}$, 
the operator $L_{\boldsymbol{\kappa}}$ is defined on functions of the 
$\mathbf z$-variable by  
\[
L_{\boldsymbol{\kappa}} := 
\tilde \mu(\exp(Y_{\boldsymbol\kappa})) - I\,. 
\]
so by the above lemma, 
\begin{equation}\label{eq:translation}
L_{\boldsymbol{\kappa}} f(z) = 
\left\{
\begin{aligned}
& f(z - (\tau, 0, \ldots, 0)) - f(z) & \text{ if } \boldsymbol\kappa = \boldsymbol\tau\,, \\ 
& [\exp(i \nu_2 z_2 ) - 1] f(z) & \text{ if } \boldsymbol\kappa = \boldsymbol\eta\,, \
\end{aligned}
\right. 
\end{equation}

The coordinates $(z_3, z_4, \ldots, z_n)$ will not play a central role, 
so for any $f  \in L^2(\R^n)$ and for any $z \in \R^n$, 
define 
\[
\begin{aligned}
& \mathbf z_3 := (z_3, z_4, \ldots, z_n) \in \R^{n-2}\,, \\ 
& f_{\mathbf z_3}(z_1, z_2) := f(z)\,.
\end{aligned}
\]   
For $j = 1, 2$, let $\mathcal F_j$ be the Fourier transform 
in the $z_j$-variable, 
so 
\[
\begin{aligned}
& \mathcal F_1 f_{\mathbf z_3}(\omega_1, z_2) := \int_\R f_{\mathbf z_3}(z_1, z_2) e^{-2\pi i \omega_1 z_1} dz_1\,, \\ 
& \mathcal F_2 f_{\mathbf z_3}(z_1, \omega_2) := \int_\R f_{\mathbf z_3}(z_1, z_2) e^{-2\pi i \omega_2 z_2} dz_2\,.  
\end{aligned}
\]

We begin with a short lemma. 
\begin{lemma}\label{lemm:f_decay}
For any $s \geq 0$ and for any $\epsilon > 0$, 
there is a constant $C_{\epsilon} > 0$ such that for any $\mathbf z \in \R^n$ 
and for any $f \in W^{s+n/2+\epsilon}(\R^n)$, the functions $f$, 
$\mathcal F_1 f$ and $\mathcal F_2 f$ are continuous on $\R^n$, and  
\begin{equation}\label{eq:f-F-decay:1}
\begin{aligned}
 & \vert f_{\mathbf z_3}(z_1, z_2) \vert \leq \frac{C_{\epsilon}}{(1 + \sum_{i=1}^n z_i^2)^{s/2}} \opnorm{f}_{s + n/2 + \epsilon}\,, \\ 
 & \vert \mathcal F_2 f_{\mathbf z_3}(z_1, \omega_2) \vert \leq \frac{C_{\epsilon}}{(1 + \omega_2^2 + \sum_{\substack{1 \leq i \leq n \\ i \neq 2}} z_i^2)^{s/2}} \opnorm{f}_{s + n/2 + \epsilon}\,.
 \end{aligned} 
\end{equation}
Similarly, for any $(\omega, z_2) \in \R^2$, for any $r \geq 0$ and for any $f \in W^{s+r+ n/2+\epsilon}(\R^n)$
\begin{equation}\label{eq:f-F-decay:2}
\begin{aligned}
& \vert f_{\mathbf z_3}(z_1, z_2) \vert \leq \frac{C_\epsilon}{(1 + z_1^2)^{r/2}(1 + \sum_{i = 2}^n z_i^2)^{s/2} } \opnorm{f}_{s +r+ n/2 + \epsilon}\,, \\ 
& \vert \mathcal F_1 f_{\mathbf z_3} (\omega_1, z_2) \vert  \leq \frac{C_\epsilon}{(1 + \omega_1^2)^{r/2}(1 + \sum_{i = 2}^n z_i^2)^{s/2} } \opnorm{f}_{r+s + n/2+\epsilon} \,.
\end{aligned}
\end{equation}
\end{lemma}
\begin{proof}   
The Sobolev embedding theorem implies  
there is a constant $C_\epsilon > 0$ such that 
\[
\begin{aligned}
\opnorm{ (I + \sum_{i=1}^n z_i^2)^{s/2} f }_{C(\R^n)} 
& \leq C_\epsilon \opnorm{(I + \sum_{i=1}^n z_i^2)^{s/2} f }_{n/2+\epsilon} \\ 
& \leq C_\epsilon \opnorm{(I+\sum_{i=1}^n z_i^2- \frac{\partial}{\partial z_i^2})^{s/2} f }_{n/2+\epsilon} \\ 
& = C_\epsilon \opnorm{f}_{s+n/2+\epsilon}  \,.
\end{aligned}
\]

The second inequality in \eqref{eq:f-F-decay:1} follows in the same way by 
applying the inverse Fourier transform $\mathcal F_2^{-1}$.  

For \eqref{eq:f-F-decay:2}, 
the Sobolev embedding theorem again gives 
a constant $C_\epsilon > 0$ such that 
\[
\begin{aligned}
\Vert &(I + \omega_1^2)^{r/2} (I + \sum_{i = 2}^n z_i^2)^{s/2} \mathcal F_1 f\Vert_{C(\R^n)}  \\ 
& \leq C_\epsilon \Vert (I - \frac{\partial^2}{\partial \omega_1^2} - \sum_{i = 2}^n\frac{\partial^2}{\partial z_i^2})^{(n/2 + \epsilon)/2} (I + \omega_1^2)^{r/2} (I+\sum_{i = 2}^n z_i^2)^{s/2}  \mathcal F_1 f \Vert_{L^2(\R^n)}\\ 
 & \leq C_\epsilon \Vert (I  + z_1^2 - \sum_{i = 2}^n\frac{\partial^2}{\partial z_i^2})^{(n/2 + \epsilon)/2} (I - \frac{\partial^2}{\partial z_1^2})^{r/2}(I + \sum_{i = 2}^n z_i^2)^{s/2} f \Vert_{L^2(\R^n)}  \\ 
& \leq C_\epsilon \opnorm{ f }_{r + s + n/2 + \epsilon} \,. 
\end{aligned}
\]
The second estimate of \eqref{eq:f-F-decay:2} follows as above.  
\end{proof}

\subsubsection{Invariant operators and cohomological equations}

For any $m \in \Z$, let $\pi_{m, \boldsymbol\tau}$ be the formal operator  
\begin{equation}\label{def:pitau}
\pi_{m, \boldsymbol\tau} f(z) := \mathcal F_1 f(\frac{m}{\tau}, z_2, \ldots, z_n)\,.  
\end{equation}
We now record a decay estimate of $\vert \pi_{m, \boldsymbol\tau}(f) \vert_s$ with respect to $m$, 
which will be used later in the splitting result, Theorem~\ref{main:split}.  
\begin{corollary}\label{cor:inv_dist}
For any $\epsilon > 0$, there is a constant $C_\epsilon > 0$ 
such that for any $s, r \geq 0$ and for any $m \in \Z$, 
the operator $\pi_{m, \boldsymbol\tau}$ satisfies 
the following estimate.  
For any $f \in W^{r+s+3n/2-1+\epsilon}(\R^n)$,  we have 
\[
 \vert \pi_{m, \boldsymbol\tau}(f) \vert_s \leq C_\epsilon (1 + \vert \frac{m}{\tau}\vert )^{-r} \opnorm{ f }_{r+s+3n/2-1+\epsilon}\,.
\]
\end{corollary}
\begin{proof}
First let $f \in \mathscr{S}(\R^n)$.  
Because $\mathcal F_1$ commutes with 
$(I - \sum_{i = 2}^n \frac{\partial^2}{\partial z_i^2} + z_i^2)^{s/2}$, 
for any $m \in \Z$, we have 
\begin{align}
\vert \pi_{m, \boldsymbol\tau}(f) \vert_s & = \Vert (I - \sum_{i = 2}^n \frac{\partial^2}{\partial z_i^2} + z_i^2)^{s/2} \mathcal F_1 f_{\mathbf z_3}(\frac{m}{\tau}, z_2)\Vert_{L^2(\R^{n-1})} \notag \\ 
& = \Vert \mathcal F_1((I - \sum_{i = 2}^n \frac{\partial^2}{\partial z_i^2} + z_i^2)^{s/2} f_{\mathbf z_3})(\frac{m}{\tau}, z_2) \Vert_{L^2(\R^{n-1})} \,.\label{eq:pi-1}
\end{align}
Then Lemma~\ref{lemm:f_decay} gives 
\[
\begin{aligned}
\vert \mathcal F_1((I - \sum_{i = 2}^n \frac{\partial^2}{\partial z_i^2} &+ z_i^2)^{s/2}  f_{\mathbf z_3})(\frac{m}{\tau}, z_2)\vert \leq \frac{C_\epsilon}{(1 + (\frac{m}{\tau})^2)^{r/2}(1 + \sum_{i = 2}^n z_i^2)^{(n -1+ \epsilon)/2} } \\
&\times \opnorm{ (I - \sum_{i = 1}^n \frac{\partial^2}{\partial z_i^2} + z_i^2)^{s/2} f_{\mathbf z_3} }_{r+3n/2-1+2\epsilon} \\ 
& \leq \frac{C_\epsilon}{(1 + (\frac{m}{\tau})^2)^{r/2}(1 + \sum_{i = 2}^n z_i^2)^{(n -1+ \epsilon)/2} } \opnorm{ f }_{s + r + 3n/2-1+ 2\epsilon}\,.
\end{aligned}
\]
So 
\begin{align}\label{eq:pi-1}
\eqref{eq:pi-1}  
& \leq C_\epsilon (1 + \vert \frac{m}{\tau}\vert )^{-r} \opnorm{ f }_{r+s+3n/2-1+2\epsilon} \notag\,.
\end{align}
\end{proof}

The next lemma shows that for any $m \in \Z$, 
$\pi_{m, \boldsymbol\tau}$ are invariant operators for 
$\tilde \mu(\exp(Y_{\boldsymbol\tau}))$ on sufficiently regular functions.  
\begin{lemma}\label{lemm:L-tau-invariant}
For any $m \in \Z$ 
and for any $\epsilon > 0$, 
\[
\pi_{m, \boldsymbol\tau} L_{\boldsymbol\tau} = 0\,.
\]
holds on $W^{n/2 + \epsilon}(\R^n)$.
\end{lemma}
\begin{proof}
The above lemma shows that for any $m \in \Z$ and any 
$f \in W^{n/2+\epsilon}(\R^n)$, $\pi_{m, \boldsymbol\tau} f$ is continuous on $\R^{n-1}$.  
Moreover, 
\[
\begin{aligned}
\pi_{m, \boldsymbol\tau}\tilde \mu(\exp(Y_{\boldsymbol{\tau}}))f(z) 
& = \pi_{m, \boldsymbol\tau} f(z - (\tau, 0, \ldots, 0))\\
& = \pi_{m, \boldsymbol\tau}f(z) \,.
\end{aligned}
\]
\end{proof}

For any $s > 2$, define 
\[
\text{Ann}_{\boldsymbol\tau} := \{f \in \mathscr{S}(\R^n) : \pi_{m, \boldsymbol\tau}(f) \equiv 0 \text{ for all }m \in \Z \}\,.
\]
\begin{proposition}\label{prop:coboundary}
For any $f \in \text{Ann}_{\boldsymbol\tau}$, the cohomological equation 
\begin{equation}\label{eq:L1P}
L_{\boldsymbol{\tau}} P = f
\end{equation} 
has a unique solution $P$ in $L^2(\R^n)$, and moreover, 
for any $\epsilon > 0$, there is a constant $C_{\epsilon} > 0$ 
such that for any $s \geq 0$ 
\[
\Vert P \Vert_s \leq \frac{C_{\epsilon}}{\tau} \Vert f \Vert_{s + n+1+ \epsilon}\,.
\]
\end{proposition}
\begin{proof}
By \eqref{eq:translation}, the cohomological equation \eqref{eq:L1P} is 
\[
P(z - (\tau, 0, \ldots, 0)) - P(z) = f(z)\,.
\] 
Clearly, there is at most one $L^2(\R^n)$ solution $P$ to the above equation.  

Define 
 \begin{equation}\label{eq:P_1}
 P(z) := \sum_{m = 1}^\infty f(z_1 + m \tau, z_2, \ldots, z_n)\,, 
 \end{equation}
 and observe that by Lemma~\ref{lemm:f_decay}, the above sum converges uniformly on compact sets and absolutely, because $f \in W^{\infty}(\R^n)$.      
So   
\[
L_{\boldsymbol{\tau}} P= f
\]
on $\R^n$\,.
 
Because $f \in \text{Ann}_{\boldsymbol{\tau}}$, 
the Poisson summation formula gives that for any $z \in \R^n$,  
\[
\sum_{m \in \Z} f(z_1 + m \tau, z_2, \ldots, z_n) = 0 
\]
By combining the above equality with \eqref{eq:P_1}, we get that 
\begin{equation}\label{eq:P_z_1-neg}
P(z) = \sum_{m = 1}^\infty f(z_1 - m\tau, z_2, \ldots, z_n)\,, 
\end{equation}
which is again convergent, by Lemma~\ref{lemm:f_decay}.  

Now we estimate $\Vert P \Vert_s$. 
By Lemma~\ref{lemm:f_decay} and formula~\eqref{eq:P_1}, 
we get that for all $z\in (\R^+\cup\{0\}) \times \R^{n-1}$,    
\[
\begin{aligned}
\vert (I + &\sum_{i= 1}^n z_i^2-\frac{\partial}{\partial z_i^2})^{s/2} P(z) \vert \\ 
& \leq \sum_{m = 0}^\infty \vert (I + \sum_{i= 1}^n z_i^2-\frac{\partial}{\partial z_i^2})^{s/2} f(z_1+m\tau, z_2, \ldots, z_n) \vert \\ 
& \leq \sum_{m = 0}^\infty 
|(I - \frac{\partial^2}{\partial (z_1+m\tau)^2} +(z_1+m\tau)^2- \sum_{i=2}^n\frac{\partial^2}{\partial z_i^2} +z_i^2)^{s/2}  \\ 
& \times f(z_1+m\tau, z_2, \ldots, z_n)\vert \\ 
& \leq C_{\epsilon} \sum_{m = 0}^\infty (1 + (z_1+m\tau)^{2} + \sum_{i = 2}^n z_i^2)^{-(1+\epsilon)} \opnorm{ f }_{s + n/2 +2+ 3\epsilon}  \\
& \leq \frac{C_{\epsilon}}{\tau} (1 + z_1^2+z_2^2)^{-(1 + \epsilon)/2} \opnorm{ f }_{s + n/2 +2+ 3\epsilon} \,.
\end{aligned}
\]

Using the \eqref{eq:P_z_1-neg}, we get by a 
completely analogous argument that for all $z \in \R^- \times \R^{n-1}$, 
\[
\vert (I - \sum_{i= 1, 2} \frac{\partial}{\partial z_i^2} +z_i^2) P(z_1, z_2) \vert\leq \frac{C_{\epsilon}}{\tau} (1 + z_1^2+z_2^2)^{-(1 + \epsilon)/2} \opnorm{ f }_{s + 3 + 2\epsilon} \,.
\]
It follows that 
\[
\opnorm{ P }_{s} \leq \frac{C_{\epsilon}}{\tau} \opnorm{ f }_{s + 3 + 2\epsilon} \,.
\]
Finally, Lemma~\ref{lemm:reduce_h1} implies the result.  
\end{proof}

Now we find a solution with Sobolev estimates to the equation 
$L_{\boldsymbol\eta} P = f\,. $ 
For any $m \in \Z$, define $\pi_{m, \boldsymbol\eta}$ to be the formal operator  
\[
\pi_{m, \boldsymbol\eta} f(z_1, z_3, \ldots, z_n) := f(z_1, \frac{2\pi m}{\nu_2}, z_3, \ldots, z_n)\,.
\] 
We get as in Corollary~\ref{cor:inv_dist} 
that for any $s \geq 0$ and $\epsilon > 0$, 
$\pi_{m, \boldsymbol\eta} : W^{s+3n/2-1+\epsilon}(\R^n) \to W^{s}(\R^{n-1})$, 
and by Lemma~\ref{lemm:f_decay} that for any $f \in W^{n/2 + \epsilon}(\R^n)$, 
 $\pi_{m, \boldsymbol\eta} f$ is continuous on $\R^{n-1}$.  
  
 As in Lemma~\ref{lemm:L-tau-invariant}, it can be immediately verified that 
for any $m \in \Z$, $\pi_{m, \boldsymbol\eta}$ is invariant for the operator $\tilde \mu(\exp(Y_{\boldsymbol\eta}))$.  
Define 
\[
\text{Ann}_{\boldsymbol\eta} := 
\left\{f \in \mathscr{S}(\R^n) : \pi_{m, \eta}(f) \equiv 0 \text{ for all }m \in \Z\right\}\,.
\] 
We have a corresponding estimate for the cohomological equation $L_{\boldsymbol{\eta}} P = f$.  

\begin{corollary}\label{coro:coboundary:L2}
For any $f \in \text{Ann}_\eta$\,, 
the equation 
\[
L_{\boldsymbol{\eta}} P = f
\]
has a unique solution $P$ in $L^2(\R^n)$, and moreover, 
For any $\epsilon > 0$, there is a constant $C_{\epsilon} > 0$ 
such that for any $s \geq 0$ 
\[
\Vert P \Vert_s \leq \frac{C_{\epsilon}}{\nu_2} \Vert f \Vert_{s +n+1+ \epsilon}\,.
\]
\end{corollary}
\begin{proof}
Writing \eqref{eq:translation} in Fourier transform, we get 
\begin{equation}\label{eq:FourierY_eta}
\mathcal F_2 L_{\boldsymbol{\eta}} f_{\mathbf z_3}(z_1, \omega_2) =  \mathcal F_2 f_{\mathbf z_3}(z_1, \omega_2 - \frac{\nu_2}{2\pi}) - \mathcal F_2 f_{\mathbf z_3}(z_1, \omega_2)\,.  
\end{equation}
Then setting $\tau = \frac{\nu_2}{2\pi}$,   
the corollary follows in the same way as Proposition~\ref{prop:coboundary}.  
\end{proof}

Next, we prove Theorem~\ref{main:transfer} for  
Schr$\ddot{\text{o}}$dinger representations.
\begin{theorem}\label{transfer-infinite}
For any $f, g \in \mathscr{S}(\R^n)$ that satisfy 
$L_{\boldsymbol{\tau}} g = L_{\boldsymbol{\eta}} f$, 
there is a solution $P \in \mathscr{S}(\R^n)$ such that 
\[
L_{\boldsymbol{\tau}} P = f \text{ and } L_{\boldsymbol{\eta}} P = g\,.
\]  
Moreover, for any $\epsilon > 0$, there is a constant 
$C_{\epsilon} > 0$ such that for any $s \geq 0$, 
\[
\Vert P \Vert_{s} \leq \frac{C_{\epsilon}}{\tau} \Vert f \Vert_{s+n+1+\epsilon}\,.
\] 
\end{theorem}
\begin{proof}   
Let $m \in \Z$.  
Because $\pi_{m, \boldsymbol\tau}$ is invariant for 
$\tilde \mu(\exp(Y_{\boldsymbol\tau}))$,  
we have that 
\[
0 \equiv \pi_{m, \boldsymbol\tau} L_{\boldsymbol{\tau}} g = \pi_{m, \boldsymbol\tau}L_{\boldsymbol{\eta}} f\,.
\]
From the formulas for $\pi_{m, \boldsymbol\tau}$ and $L_{\boldsymbol\eta}$, 
see \eqref{def:pitau} and \eqref{eq:translation} respectively, we get  
\[
[L_{\boldsymbol\eta}, \pi_{m, \boldsymbol\tau}] = 0\,.  
\]
Moreover, for any $(z_2, \ldots, z_n) \in \R^{n-1}$\,,
\[
 0 = L_{\boldsymbol{\eta}} \pi_{m, \boldsymbol\tau}f (z_2, \ldots, z_n) = [\exp( i \nu_2 z_2) - 1] \mathcal F_1f(\frac{m}{\tau}, z_2, \ldots, z_n) \,.  
 \]
 So we get that off a countable set of $z_2 \in \R$, 
\[
\mathcal F_1f(\frac{m}{\tau}, z_2, \ldots, z_n) = 0\,.  
\]
Lemma~\ref{lemm:f_decay} shows that $\mathcal F_1f$ is continuous, 
which implies that $\pi_{m, \boldsymbol\tau} f \equiv 0$\,.  
Because $m \in \Z$ was arbitrary, we conclude that $f \in \text{Ann}_{\boldsymbol\tau}$.  

 Proposition~\ref{prop:coboundary} now implies 
 there is a unique function $P$ in $L^2(\R^n)$ that is a solution to 
 \[
L_{\boldsymbol{\tau}} P = f\,,
 \]
 and for any  $\epsilon > 0$, there is a constant $C_\epsilon > 0$ such that 
 \[
 \Vert P \Vert_{s} \leq \frac{C_{\epsilon}}{\tau} \Vert f \Vert_{s+n+1+\epsilon} \,.
 \]
 
Finally, because $[Y_{\boldsymbol\eta}, Y_{\boldsymbol\tau}] = 0$, 
we use $L_{\boldsymbol{\tau}} g = L_{\boldsymbol{\eta}} f$ and get  
 \[
 L_{\boldsymbol{\tau}} g = L_{\boldsymbol{\eta}} L_{\boldsymbol{\tau}} P =  L_{\boldsymbol{\tau}}  L_{\boldsymbol{\eta}} P\,.
 \]
So $L_{\boldsymbol{\tau}}(g-  L_{\boldsymbol{\eta}} P) = 0$, and because $g -  L_{\boldsymbol{\eta}}P \in L^2(\R^n)$ 
it follows by ergodicity that 
 \[
g =  L_{\boldsymbol{\eta}} P
 \]
in $L^2(\R^n)$. 
\end{proof}

Now we will prove Theorem~\ref{main:split} in the case of  
Schr$\ddot{\text{o}}$dinger representations.  
Recall from Lemma~\ref{lemm:basic_oper} that $\nu_2 \neq 0$.  
\begin{theorem}\label{split:infinite} 
For any $f, g, \phi \in \mathscr{S}(\R^n)$ that 
satisfy $L_{\boldsymbol{\eta}} f - L_{\boldsymbol{\tau}} g = \phi$, 
there exists a nonconstant function $P \in \mathscr{S}(\R^n)$ such that 
the following holds.  
For any $s \geq 0$ and for any $\epsilon > 0$, there is a constant 
$C_{s, \epsilon} > 0$ such that 
\[
\begin{aligned}
& \|g - L_{\boldsymbol{\eta}} P\|_s \leq C_{s, \epsilon} (\tau^{-1} + \tau^{\epsilon}) \| \phi\|_{s + 5n/2+1 + \epsilon}\,, \\ 
& \|f - L_{\boldsymbol{\tau}} P\|_s \leq \frac{C_{s, \epsilon}}{\nu_2} (1 + \tau^{1+\epsilon})  \| \phi\|_{s + 5n/2+1 + \epsilon} \,, \\ 
& \| P \|_s \leq C_{s, \epsilon} (\tau^{-1} + \tau^{\epsilon}) (\| f\|_{s + 5n/2+1 + \epsilon} + \| g\|_{s + 5n/2+1 + \epsilon})\,.
\end{aligned}
\]
\end{theorem}
\begin{proof}
Notice that if $f = g = 0$, then $\phi = 0$, and the above statement holds trivially.  
Without loss of generality, we assume that $f \neq 0$.  

Let $\psi \in \mathscr{S}(\R)$ be any function such that 
$\hat \psi \in C_c^\infty([-\frac{1}{2\tau}, \frac{1}{2 \tau}])$ and $\hat \psi(0) = 1$.  
For each $m \in \Z$, define the functional $\Pi_{m, \boldsymbol\tau}$ on $L^2(\R^{n-1})$
by  
\begin{equation}\label{eq:def-Pin}
\Pi_{m, \boldsymbol\tau} F(z_2, \ldots, z_n) = e^{2\pi i z_1 m/\tau} \psi(z_1) F(z_2, \ldots, z_n)\,.
\end{equation}
\begin{lemma}\label{lemm:Pi}
For any $s \in 2 \N$, for any $m \in \Z$ 
and for any $F \in W^s(\R^{n-1})$,  
there is a constant $C_s > 0$ such that 
\[
\opnorm{ \Pi_{m, \boldsymbol\tau} F }_s \leq C_s \Vert (I - \frac{\partial^2}{\partial z_1^2})^{s/2} \psi \Vert_{L^\infty(\R)} \sum_{k = 0}^{s/2} (1+\vert\frac{m}{\tau}\vert)^{2k} \vert F \vert_{s-2k} \,.
\]
\end{lemma}
\begin{proof}
Because $\psi$ is supported on $[-1/2, 1/2]$, we have 
\begin{align}
\opnorm{ \Pi_{m, \tau} & F }_s = 
\Vert (I + \sum_{i = 1}^n z_i^2-\frac{\partial^2}{\partial z_i^2})^{s/2} (e^{- 2\pi i z_1m/\tau} \psi F)\Vert_{L^2(\R^n)} \notag \\ 
& \leq C_s \Vert (- \frac{\partial^2}{\partial z_1^2}+(I +\sum_{i = 2}^n z_i^2-\frac{\partial^2}{\partial z_i^2}))^{s/2} (e^{- 2\pi i z_1 m/\tau} \psi F)\Vert_{L^2(\R^n)} \,.\label{eq:Pi-1}
\end{align} 
Then because $- \frac{\partial^2}{\partial z_1^2}$ and 
$(I +\sum_{i = 2}^n z_i^2-\frac{\partial^2}{\partial z_i^2})$ commute, 
the triangle inequality gives 
\begin{align}
& \eqref{eq:Pi-1} \leq C_s \sum_{k = 0}^{s/2} \Vert (\frac{\partial^2}{\partial z_1^2})^k (I +\sum_{i = 2}^n z_i^2-\frac{\partial^2}{\partial z_i^2})^{s/2-k} (e^{- 2\pi i z_1 m/\tau} \psi F) \Vert_{L^2(\R^n)} \notag \\ 
& \leq C_s \sum_{k = 0}^{s/2} \Vert  (\frac{\partial^2}{\partial z_1^2})^k e^{- 2\pi i z_1m/\tau} \psi \Vert_{L^\infty(\R)} \Vert (I +\sum_{i = 2}^n z_i^2-\frac{\partial^2}{\partial z_i^2})^{s/2-k} F \Vert_{L^2(\R^{n-1})} \notag  
\end{align}  
\end{proof}

Note that $\Pi_{m, \boldsymbol\tau}$ depends on $\psi$, 
so formally define the operator $R_\psi$ on $L^2(\R^n)$ by 
\begin{equation}\label{eq:def-R}
R_\psi := I - \sum_{m \in \Z} \Pi_{m, \boldsymbol\tau} \pi_{m, \boldsymbol\tau}\,.
\end{equation}
Over the next two lemmas, we describe properties of $R_\psi$.  
\begin{lemma}\label{lemm:R}
For any $s \geq 0$ and for any $\epsilon > 0$, 
there is a constant $C_{s, \epsilon} > 0$ such that 
for any nonzero $f \in W^{s+n/2+\epsilon}(\R^n)$, we can choose $\psi$ such that 
$R_\psi f \neq 0$ and 
\[
\Vert R_\psi f\Vert_s \leq C_{s, \epsilon} (1 +  \tau^{s + 1 + \epsilon})  \Vert f \Vert_{s + 3n/2 + \epsilon}\,.
\]
\end{lemma}
\begin{proof}
We first claim that we can choose $\psi$ such that $R_{\psi} f \neq 0$, and for some universal constant $C_s^{(0)} > 0$, 
\[
\Vert (I - \frac{\partial^2}{\partial z_1^2})^{s/2} \psi \Vert_{L^\infty(\R)} \leq C_s^{(0)} (1 + \tau^s)\,.
\]
Fix $\psi$.  So for some $C_s^{(0)} > 0$, the above estimate holds.  If $R_\psi f \neq 0$, then the claim is holds, so suppose that $R_\psi f = 0$.  
Hence,   
\[
f(z) = \psi(z_1) \sum_{m \in \Z} \exp(2\pi i z_1 m/\tau) \mathcal F_1 f(\frac{m}{\tau}, z_2, \ldots, z_n)\,. 
\]
So we can perturb $\hat\psi$ to a 
function $\hat \psi_0 \in C_c^\infty([-\frac{1}{2\tau}, \frac{1}{2\tau}])$ 
satisfying $\hat \psi_0(0) = 1, \psi_0 \neq \psi$ 
and $R_\psi f \neq 0$, where also 
\begin{equation}\label{eq:psi_0-upper}
\Vert (I - \frac{\partial^2}{\partial z_1^2})^{s/2} \psi_0 \Vert_{L^\infty(\R)} \leq (C_s^{(0)} + 1) (1 + \tau^s)\,. 
\end{equation}
This proves the claim.  

Now say $s \in \N$ is even, and let $R_\psi f \neq 0$, 
where $\psi$ satisfies \eqref{eq:psi_0-upper}.  
By the triangle inequality and Lemma~\ref{lemm:Pi}, 
we get a constant $C_s^{(1)} > C_s^{(0)} + 1$ such that 
\begin{align}
\opnorm{ R_\psi f }_s & \leq \opnorm{ f }_s +  \sum_{m \in \Z} \opnorm{ \Pi_{m, \tau} \pi_{m, \tau} f }_s \notag \\ 
& \leq \opnorm{ f }_s + C_s^{(1)} (1 + \tau^s) \sum_{m \in \Z} \sum_{k = 0}^{s/2}  (1+\vert\frac{m}{\tau}\vert)^{2k} \vert \pi_{m, \tau} f \vert_{s-2k}\,. \label{eq:R2} 
\end{align}
By Corollary~\ref{cor:inv_dist}, there is a constant 
$C_\epsilon > 0$ such that for any $m \in \Z$, 
\begin{align}
(1+\vert\frac{m}{\tau}\vert)^{2k} \vert \pi_{m, \tau} f \vert_{s-2k}  & \leq C_{\epsilon} 
(1+\vert\frac{m}{\tau}\vert)^{2k} (1 + \vert \frac{m}{\tau}\vert)^{-(2k + 1 + \epsilon)} \opnorm{f }_{s + 3n/2 + \epsilon} \notag \\ 
& \leq C_\epsilon (1 +  \tau)^{s+ 1 + \epsilon} (1 + \vert m\vert)^{-(1 + \epsilon)} \opnorm{ f }_{s + 3n/2 + \epsilon}\,. \notag
\end{align}

Hence, there is a constant $C_{s, \epsilon} > 0$ such that 
\[
\begin{aligned} 
\eqref{eq:R2} & \leq \opnorm{ f }_s  + C_{s, \epsilon}(1 +  \tau)^{s+ 1 + \epsilon} \sum_{m \in \Z} (1 + m^2)^{-(1+\epsilon)/2} \opnorm{ f }_{s+3n/2+\epsilon} \\ 
& \leq C_{s, \epsilon} (1 +  \tau)^{s + 1 + \epsilon} \opnorm{ f }_{s+3n/2+\epsilon} \,.
\end{aligned}
\]
Because $R_\psi$ is a linear operator, Lemma~\ref{lemm:reduce_h1} gives the estimate for even integers  $s \geq 0$.    
The lemma now follows by interpolation.  
\end{proof}

Next we show that the operator $R_\psi$ is a projection 
into $\text{Ann}_{\boldsymbol\tau}$  
and it commutes with $L_{\boldsymbol{\eta}}$.  
\begin{lemma}\label{lemm:Rf-coboundary}
Let $\psi$ be as in the previous lemma.  Then  
\[
R_\psi: \mathscr{S}(\R^n) \to \text{Ann}_{\boldsymbol\tau}\,,
\]
and 
\[
R_\psi L_{\boldsymbol{\tau}} = L_{\boldsymbol{\tau}} \,, \quad [R_\psi, L_{\boldsymbol{\eta}}] = 0
\]
on $\mathscr{S}(\R^n)$.  
\end{lemma}
\begin{proof}
Let $f \in \mathscr{S}(\R^n)$.   
By the previous lemma, $R_{\psi} f \in \mathscr{S}(\R^n)$, so 
we need to show that $R_{\psi} f$ is in the kernel of every $\pi_{m, \boldsymbol\tau}$.  
Using the property that $\hat \psi$ is supported on the interval $[-\frac{1}{2\tau}, \frac{1}{2\tau}]$ and $\hat \psi(0) = 1$, we get that for any $m \in \Z$,  
\begin{align}
\pi_{m, \tau} R_\psi f & = \pi_{m, \tau}f - \sum_{k \in \Z} \pi_{m, \tau} \Pi_{k, \tau} \pi_{k, \tau}f \notag \\ 
& = \pi_{m, \tau}f - \pi_{m, \tau} \Pi_{m, \tau} \pi_{m, \tau}f \notag \\ 
& = \pi_{m, \tau}f - \pi_{m, \tau}f \notag \\ 
& = 0\,. \label{eq:pi_mR=0}
\end{align}
This implies $R_{\psi} :\mathscr{S}(\R^n) \to Ann_{\boldsymbol\tau}$.  

By Lemma~\ref{lemm:L-tau-invariant}, for any $m \in \Z$, 
 $\pi_{m, \boldsymbol\tau} L_{\boldsymbol\tau}  = 0$.  
We have 
\[
\begin{aligned}
R_\psi L_{\boldsymbol{\tau}} & = (I - \sum_{m \in \Z} \Pi_{m, \tau} \pi_{m, \tau}) L_{\boldsymbol{\tau}} \\ 
& = L_{\boldsymbol{\tau}}\,.
\end{aligned}
\]

Finally, we prove that $[L_{\boldsymbol{\eta}}, R_\psi] = 0$.  
We have 
\[
\begin{aligned}
\Pi_{m, \boldsymbol\tau}  \pi_{m, \boldsymbol\tau} L_{\boldsymbol{\eta}} f(z) & = 
\exp(2\pi i z_1 m/\tau) \psi(z_1) [\pi_{m, \boldsymbol\tau} L_{\boldsymbol{\eta}} f](z_2, \ldots, z_n) \\ 
& = \exp(2\pi i z_1 m/\tau) \psi(z_1) \mathcal F_1[L_{\boldsymbol{\eta}} f](\frac{m}{\tau}, z_2, \ldots, z_n) \\ 
& = (\exp(i \nu_2 z_2) - 1) \exp(2\pi i z_1 m/\tau) \psi(z_1) \mathcal F_1f(\frac{m}{\tau}, z_2, \ldots, z_n) \\ 
& = L_{\boldsymbol{\eta}} \exp(2\pi i z_1 m/\tau) \psi(z_1) \mathcal F_1 f(\frac{m}{\tau}, z_2, \ldots, z_n) \\ 
& = L_{\boldsymbol{\eta}} \Pi_{m, \boldsymbol\tau} \mathcal F_1 f(\frac{m}{\tau}, z_2, \ldots, z_n) \\ 
& = L_{\boldsymbol{\eta}} \Pi_{m, \boldsymbol\tau} \pi_{m, \boldsymbol\tau} f\,.
\end{aligned}
\]
This proves $[L_{\boldsymbol{\eta}}, R_\psi] = 0\,,$ 
and finishes the proof of the lemma.  
\end{proof}

Now we prove Theorem~\ref{split:infinite}.  
Let 
\[
\phi = L_{\boldsymbol{\eta}} f - L_{\boldsymbol{\tau}} g
\] 
be as in the theorem, and recall from the beginning of its proof that we take $f \neq 0$.  
By Lemmas~\ref{lemm:R} and \ref{lemm:Rf-coboundary}, 
we can choose $\psi$ such that there is a 
nonconstant function $P$ that is a solution to 
$R_\psi f = L_{\boldsymbol{\tau}} P$, and for a fixed constant $C_s^{(1)} > 0$, 
\[
\Vert (I - \frac{\partial^2}{\partial z_1^2})^{s/2} \psi \Vert_{L^\infty(\R)} \leq C_s^{(1)}\,.
\]
In particular, Lemma~\ref{lemm:Rf-coboundary} implies 
\begin{equation}\label{eq:R_phi}
\begin{aligned}
R_\psi \phi & = R_\psi L_{\boldsymbol{\eta}} f - R_\psi L_{\boldsymbol{\tau}} g \\ 
& = L_{\boldsymbol{\eta}} R_\psi f - L_{\boldsymbol{\tau}}  g \\ 
& = L_{\boldsymbol{\eta}} L_{\boldsymbol{\tau}} P - L_{\boldsymbol{\tau}}  g\\ 
& = L_{\boldsymbol{\tau}} (L_{\boldsymbol{\eta}} P - g)\,.
\end{aligned}
\end{equation}
Then by Proposition~\ref{prop:coboundary} and by 
Lemmas~\ref{lemm:R} and \ref{lemm:sub-laplace}
we get that for any $s \geq 0$, and for any $\epsilon  > 0$, 
there is a constant $C_{s, \epsilon} > 0$ such that 
\begin{align}
\Vert L_{\boldsymbol{\eta}} P - g \Vert_s & = \Vert L_{\boldsymbol{\eta}} P - g \Vert_s \\ 
 & \leq \frac{C_{s, \epsilon}}{\tau} \Vert R_\psi \phi \Vert_{s +n+1 + \epsilon} \notag \\ 
& \leq C_{s, \epsilon}  (\tau^{-1} + \tau^{s+\epsilon}) \Vert \phi \Vert_{s +5n/2+1 + 2\epsilon} \,. \label{L-eta-cocycle-1}
\end{align}

To estimate $\Vert L_{\boldsymbol{\tau}} P - f \Vert_s$, 
because $R_\psi f = L_{\boldsymbol{\tau}} P$ and by Lemma~\ref{lemm:sub-laplace}, 
\begin{align}
\Vert L_{\boldsymbol{\tau}} P - f \Vert_s & = \Vert L_{\boldsymbol{\tau}} P -  f \Vert_s \notag \\
&  = \Vert L_{\boldsymbol{\tau}} P - R_\psi f  + R_\psi f - f\Vert_s \notag  \\ 
& = \Vert (R_\psi -I) f\Vert_s\,. \label{eq:L1P-f}
\end{align}
Notice that by Lemma~\ref{lemm:Rf-coboundary}, 
\[
(R_\psi  - I) L_{\boldsymbol{\tau}} g = 0\,.
\]
Then using $L_{\boldsymbol{\eta}} f - L_{\boldsymbol\tau} g = \phi$, 
we get 
\begin{equation}\label{eq:(R_I)L_eta}
\begin{aligned}
(R_\psi  - I) L_{\boldsymbol{\eta}} f  & = (R_\psi  - I) \phi + (R_\psi  - I) L_{\boldsymbol{\tau}} g \\ 
& = (R_\psi  - I) \phi\,.
\end{aligned}
\end{equation}
Then by Lemma~\ref{lemm:Rf-coboundary} again, we get 
\[
L_{\boldsymbol{\eta}} (R_\psi  - I) f = (R_\psi  - I) \phi\,.
\]
We conclude by Corollary~\ref{coro:coboundary:L2} and Lemmas~\ref{lemm:R} and \ref{lemm:sub-laplace} that 
\begin{align}
\eqref{eq:L1P-f} & \leq \frac{C_{s, \epsilon}}{\nu_2} \Vert (R_\psi  - I) \phi \Vert_{s + n+1 + \epsilon} \notag \\ 
& \leq \frac{C_{s, \epsilon}}{\nu_2} (1 + \tau^{s + 1+\epsilon}) \Vert \phi \Vert_{s+5n/2+1+2\epsilon} \,. \label{L-eta-cocycle-2}
\end{align}

Finally, because $L_{\boldsymbol{\tau}} P = R_\psi f$, 
Proposition~\ref{prop:coboundary} and Lemma~\ref{lemm:R} give 
\begin{align}
\Vert P \Vert_s & \leq \frac{C_{s, \epsilon}}{\tau} \Vert R_\psi  f \Vert_{s + n+1+ \epsilon} \notag \\ 
& \leq C_{s, \epsilon} (\tau^{-1} + \tau^{s + \epsilon})  \Vert f \Vert_{s + 5n/2+1 + 2\epsilon} \notag \,. \notag 
\end{align}

At the start of the proof of Theorem~\ref{split:infinite}, we assumed $f \neq 0$.  
If we instead chose $g \neq 0$, then by first applying the Fourier transform, the same argument proves the above estimates in terms of $\Vert \phi \Vert_{s+5n/2+1+2\epsilon}$ and 
\[
\begin{aligned}
\Vert P \Vert_s & \leq C_{s, \epsilon} (\tau^{-1} + \tau^{s + \epsilon})  \Vert g \Vert_{s + 5n/2+1 + 2\epsilon}  \,.
\end{aligned}
\]
This completes the proof of Theorem~\ref{split:infinite}.  
\end{proof}

\begin{proof}[Proof of Theorem~\ref{main:transfer}]
The regular representation of $\sH$ on $L^2(M)$ decomposes as 
\[
L^2(M) = \mathbb C \langle 1 \rangle \oplus \bigoplus_{\mathbf m \in \Z^{2n} \setminus\{0\}} \mathcal P_{\mathbf m} \oplus \bigoplus_{h \in \Z} \mathcal P_h\,,
\]
where each $\mathcal P_{\mathbf m}$ is 
an abelian representation of $\R^{2n}$ equivalent to a character given by \eqref{eq:finite_rep}, and each $\mathcal P_h$ is equivalent to a countable collection of Schr$\ddot{\text{o}}$dinger representations $\mu_h$ of $\sH$ on $L^2(\R^n)$ given by \eqref{eq:schrodinger}.  
The subspace of zero-average functions in $L^2(M)$ is denoted $L_0^2(M)$, which therefore decomposes as 
\[
\begin{aligned}
L_0^2(M) & = \bigoplus_{\mathbf m \in \Z^{2n} \setminus\{0\}} \mathcal P_{\mathbf m} \oplus \bigoplus_{h \in \Z} \mathcal P_h \\ 
& = L_0^2(\T^{2n}) \oplus \bigoplus_{h \in \Z} \mathcal P_h \,.
\end{aligned}
\]
As indicated in Section~\ref{sect:rep_spaces}, vector fields in $\h$ split into the unitary components in the above Hilbert space.  
Then the decomposition of the Sobolev space of $s$-differentiable, zero-average functions is  
\begin{equation}
W_0^s(M) = W_0^s(\T^{2n}) \oplus \bigoplus_{h \in \Z} W^s(\mathcal P_h)\,,  
\end{equation}
where $W_0^s(\T^{2n})$ and $W^s(\mathcal P_h)$ are $s$-order Sobolev spaces on the torus $\T^{2n}$ and of the representation $\mathcal P_h$, respectively.  

Now in Theorem~\ref{main:transfer}, we are given zero-average functions $f, g \in C^\infty(M)$ that satisfy $L_{\boldsymbol \tau} g = L_{\boldsymbol \eta} f$, and we aim to find a solution $P \in C^\infty(M)$ such that 
\[
L_{\boldsymbol \tau} P = f \,, \quad L_{\boldsymbol \eta} P = g\,.
\]
Write 
\[
f = f_t \oplus \bigoplus_{h \in \Z} f_h \,, \quad g = g_t \oplus \bigoplus_{h \in \Z} g_h\,,
\]
where $f_t, g_t \in W_0^\infty(\T^{2n})$ and for each $h \in \Z$, $f_h, g_h \in \mathcal P_h$.  
Proposition~\ref{prop:finite-reps-cocycle} and Theorem~\ref{transfer-infinite} give smooth solutions $P_t \in W_0^\infty(\T^{2n})$ and $\{P_h\}_{h \in \Z} \subset W^\infty(\mathcal P_h)$ satisfying the estimate Theorem~\ref{main:transfer} in the finite and infinite dimensional representations, respectively.  
Define $P \in C^\infty(M)$ by 
\[
P := P_t \oplus \bigoplus_{h \in \Z} P_h\,.
\]
So there exists a constant $C_{s, \epsilon}:= C_{s, \epsilon, \boldsymbol{\tau}, \boldsymbol{\eta}} > 0$ such that 
\begin{equation}\label{eq:main_pf:0}
\begin{aligned}
\Vert P \Vert_{W^s(M)}^2 & = \Vert P_t \Vert_{s}^2 + \sum_{h \in \Z} \Vert P_h \Vert_{s}^2  \\ 
& \leq C_{s, \epsilon} (\Vert f_t \Vert_{s + 2\gamma} + \Vert g_t \Vert_{s + 2 \gamma})^2 + C_{s, \epsilon} \sum_{h \in \Z} \Vert f_h \Vert_{s + n + 1 + \epsilon}^2  \\ 
& \leq C_{s, \epsilon} (\Vert f \Vert_{s + \max\{2\gamma, n + 1 + \epsilon\}} + \Vert g \Vert_{s + 2 \gamma})^2\,.  
\end{aligned}
\end{equation}

\end{proof}

\begin{proof}[Proof of Theorem~\ref{main:split}]
This follows from Proposition~\ref{prop:finite-reps} 
and Theorem~\ref{split:infinite} as in the proof 
of Proof of Theorem~\ref{main:transfer}. 
\end{proof}

\section{Proof of Theorem \ref{conj:conjugacydis}}\label{s:splitVF}

We fix now $Y_{\boldsymbol \eta}$ and $Y_{\boldsymbol \eta}$ with $\bdtau\cdot \bdeta=0$ and $\bdtau$ and $\bdeta$ Diophantine, as in the main setting. We denote by $\rho$ the $\mathbb Z^2$ action generated by $Y_\bdtau$ and $Y_\bdeta$ as described in \eqref{def:Z^2-action} in Section~\ref{setting}. 

In this section we prove Theorem \ref{conj:conjugacydis}. 
We will apply here similar method which was applied  in \cite{DK_UNI}. The method consists in taking successive iterations and adjustment of parameter $\lambda$ at each step.  The procedure is outlined in a general theorem which was proved in \cite{D_IFT}. There, a set of conditions in cohomology is given, which imply  transversal local rigidity of a finite dimensional family of Lie group actions. This general theorem was then used in \cite{D_GH} to obtain transversal local rigidity of certain $\mathbb R^2$ actions on 2-step nilmanifolds. Even though we have a similar situation here, we cannot unfortunately use the general theorem from \cite{D_IFT} because that theorem is for \emph{Lie group} actions, and here we have a \emph{discrete group} action. This is the only difference though, the method of successive iterations is completely parallel to that used in the above mentioned papers.

We write the proof of Theorem \ref{conj:conjugacydis} here in the case the manifold is the 5-dimensional Heisenberg nilmanifold, that is in the case $n=2$. This is the lowest dimensional case in which our result holds. We chose to present the proof for concrete $n$ for the benefit of the reader because computations are more clear and notations are simpler. Otherwise, the proof is clearly completely parallel for any $n\ge 2$. We stress the points in computation of cohomology where dimension matters, and how it affects the computation.

We will first compute in Section~\ref{constantdis} the cohomology with coefficients in constant vector fields (i.e. in the lie algebra $\frak h$) for the action $\rho$ . Then we describe in Section~\ref{family} the finite dimensional family $\rho^{\la}$ of algebraic actions to which this action belongs, where $\rho^0= \rho$. This family is completely determined by the cocycles (with values in $\frak h$) over $\rho$. Then we move on to analyse the conjugacy operator and the commutator operator in Sections~\ref{commop} and \ref{conjop} and their linarised operators. The linearisations of these two operators are corresponding to  the first and the second coboundary operators for the cohomology  over  $\rho$ with coefficients in smooth vector fields $\vect$. Using the results from the previous part of the paper (specifically Theorem \ref{main:split}), we show in Section~\ref{cohVF} that this cohomology sequence splits and that the first cohomology with coefficients in $\vect$ is the same as the cohomology with coefficients in $\frak h$. This allows us to prove Theorem  \ref{conj:conjugacydis} by showing convergence of successive iterations in Section~\ref{conv}.  

For a vector field $H\in \vect$ we denote by $H_c$ its component in the center direction and by $H_T$ the remainder, that is the component of $H$ in the off-center directions. We  denote by $\Ave(H)$ the constant vector field (i.e. an element in $\frak h$) which is obtained by taking the average of $H$ with respect to the Haar measure. 

For two vector fields $F, G\in \vect$ we use the notation \\$\|F, G\|_r:= \max \{\|F\|_r, \|G\|_r\}$, where $\|\cdot \|_r$ denotes the $C^r$-norm.

\subsection{Constant cohomology for the discrete time action}\label{constantdis}
We have
\[
\begin{aligned}
D e^{Y_\bdtau }(X_1) = X_1\,, \ \ D e^{Y_\bdtau}(X_2) = X_2, \,\,\, 
  D e^{Y_\bdtau}(Z) = Z\,.
\end{aligned} 
\]

Furthermore

\[
\begin{aligned}
e^{Y_\bdtau} e^{t \Lambda_1} & = (e^{Y_\bdtau} e^{t \Lambda_1} e^{-Y_\bdtau}) e^{Y_\bdtau} \\ 
& = \exp(e^{\text{ad}_{Y_\bdtau}} t \Lambda_1) e^{Y_\bdtau} 
= \exp(t \Lambda_1 + t [Y_\bdtau, \Lambda_1]) e^{Y_\bdtau}\\ 
& = \exp(t (\Lambda_1 + \tau_1Z)) e^{Y_\bdtau}\,.
\end{aligned}
\]
Therefore, for a constant vector field $$H= h_1X_1+ h_2X_2+ h_3\L_1+ h_4\L_2+ h_5Z,$$ where $h_i$ are constants, we have 
$$D e^{Y_\tau}(H)= h_1X_1+ h_2X_2+ h_3\L_1+ h_4\L_2+ (h_5-h_3\tau_1-h_4\tau_2)Z.$$
Another way to write this is 
$$D e^{Y_\bdtau}(H)= H+ [Y_\bdtau, H].$$
Similar computation can be done for $D e^{Y_\bdeta}$. In the matrix form we have:

\[
D e^{Y_\bdtau} = 
\begin{pmatrix}
1 &  0 & 0 & 0 & 0 \\ 
0 & 1 & 0 & 0 & 0 \\ 
0 & 0 & 1 & 0 & 0 \\ 
0 & 0 & 0 & 1 & 0 \\ 
0 & 0 & \tau_1 & \tau_2 & 1 \\ 
\end{pmatrix}\,, \ \ 
D e^{Y_\bdeta} = 
\begin{pmatrix} 
1 &  0 & 0 & 0 & 0 \\ 
0 & 1 & 0 & 0 & 0 \\ 
0 & 0 & 1 & 0 & 0 \\ 
0 & 0 & 0 & 1 & 0 \\ 
-\eta_1 & -\eta_2 & 0 &0 & 1 \\ 
\end{pmatrix}\,.
\]
A pair of constant vector fields $(F, G)\in \frak h\times \frak h$ is a {cocycle} over the $\Z^2$ action generated by $Y_\bdtau$ and $Y_\bdeta$, if 

\begin{equation*}\label{cocdis}
(D e^{Y_\bdtau}-Id) G= (D e^{Y_\bdeta}-Id) F
\end{equation*}
which implies 
\begin{equation}\label{cocdis2}
[Y_\bdtau, G]= [Y_\bdeta,  F]
\end{equation}
Because $H$ is 2-step nilpotent, this condition is only on the off-center coordinates of $F$ and $G$. More precisely
if  $F= f_1X_1+ f_2X_2+ f_3\L_1+ f_4\L_2+ f_5Z$ and $G=g_1X_1+ g_2X_2+ g_3\L_1+ g_4\L_2+ g_5Z$, then \eqref{cocdis2} implies that $g_i$ and $f_i$ for $i=1,2,3,4$ are satisfying the relation
\begin{equation}\label{cocreldis}
\tau_1g_3+\tau_2g_4+\eta_1f_1+\eta_2f_2=0
\end{equation}
Since constants $f_5$ and $g_5$ are arbitrary, so the space of constant cocycles has dimension 9. 

 A pair of constant vector fields $(F, G)$ is a coboundary, if 
 \begin{equation*}\label{cobdis}
(D e^{Y_\bdtau }-Id)H= F, \,\,\,
(D e^{Y_\bdeta}-Id) H= G
\end{equation*}
that is, if 
\begin{equation}\label{cobdis2}
[Y_\bdtau, H]= F, \,\,\,
[Y_\bdeta, H]= G
\end{equation}
This implies that off-center coordinates of both $F$ and $G$ must be zero, and the center ones must satisfy certain relations. More precisely, 
if $H=  h_1X_1+ h_2X_2+ h_3\L_1+ h_4\L_2+ h_5Z$, the equations \eqref{cobdis2} imply that 
\begin{equation}\label{cobreldis}
f_5= \tau_1h_3+\tau_2h_4, \,\,\,
g_5= -\eta_1h_1-\eta_2h_2
\end{equation}
and these equations always have solutions for coefficients of $H$.

This implies that the first cohomology is  7 dimensional, and each cohomology class is represented by cocycles $(F, G)$  of the following form
 $F= f_1X_1+ f_2X_2+ f_3\L_1+ f_4\L_2$ and $G=g_1X_1+ g_2X_2+ g_3\L_1+ g_4\L_2$, where the coefficients $g_i$ and $f_i$ for $i=1,2,3,4$ satisfy the relation \eqref{cocreldis}.

 It is clear from the above computation that the dimension of the constant cohomology over the action  generated by $Y_\bdtau$ and $Y_\bdeta$ in the case the manifold is $2n+1$-dimensional Heisenberg nilmanifold, is parallel to what we wrote above in the case $n=2$ and that the resulting cohomology has dimension $4n-1$.
 \color{black}
 
\subsection{The finite dimensional family of $\mathbb Z^2$ algebraic actions}\label{family}

For easier notation, in the rest of the paper we let $Y_1=Y_\bdtau$ and $Y_2=Y_\bdeta$.  In what follows we will use  the fact that  in $\frak h$ we have $\exp(X+Y+ \half[X, Y])= \exp(X)\exp(Y)$. In the remainder of the paper the brackets $[\cdot, \cdot]$  denote the bracket in the Lie algebra $\frak h$ 
so each bracket which appears as a result has a vector field in the $Z$ direction only.

We  define  now a 9-dimensional  family of $\mathbb Z^2$ actions $\rho^\la$ on $M=\Gamma\setminus H$ generated by the following maps :
\begin{equation}\label{fam}
y_i^\la(x)=x\cdot \exp (Y_i+F_i^\la)
\end{equation}
for $i=1,2$, subject to the commutativity relation $y_1^\la\circ y_2^\la=y_2^\la\circ y_1^\la$, where $F_i^\la\in \frak h$. In particular, at parameter $\la$ equal to 0, $F_i^0=0$, and $y_i^0=y_i$,  where $y_1=\exp Y_1$ and $y_2=\exp Y_2$ are generators of our original action $\rho$.

The following lemma is a simple computation.
\begin{lemma}
The two maps $y_1^\la$ and $y_2^\la$ commute if and only if one of the following equivalent conditions hold:
\begin{enumerate}
\item  $[Y_1 +F^\la_1, Y_2+ F^\la_2]=0$.
\item 
$[Y_1, F^\la_2]-[Y_2, F^\la_1]+ [F^\la_1, F^\la_2]=0.$
\end{enumerate}
\end{lemma}
If $F_i^\la= f_i^1X_1+ f_i^2X_2+ f_i^3\L_1+ f_i^4\L_2+f_i^5Z$
the coefficients $f_i^k\in \mathbb R$ for $i=1,2$ and $k= 1,2,3,4, 5$, commutativity implies that the coefficients are  subject to relation 

\begin{equation}\label{comrel}
\tau_1f_2^3+\tau_2f_2^4+\eta_1f_1^1+\eta_2f_1^2=f_2^1f_1^3+f_2^2f_1^4 -f_2^3f_1^1-f_2^4f_1^2.
\end{equation}

Therefore, parameter $\la$ is understood here as a vector in the  9-dimensional space: $$\{f_i^k \in \mathbb R, i=1,2,\, k= 1,2,3,4, 5,\, \mbox{ subject to relation}\,\, \eqref{comrel}\}.$$

Within this 9-dimensional family of actions we impose identifications via conjugacies obtained by constant vector fields. More precisely, if
$$y_i^\la(x)= x\cdot \exp(Y_i+ F_i^\la)$$ 
and for some $H\in \frak h$ and $h(x):=x\cdot \exp H$ we have 
$$h\circ y_i^\la(x) = y_i\circ h,$$
then 
$$x\cdot \exp(Y_i+F_i^\la)\exp H= x\cdot \exp H\exp Y_i.$$
This implies 
$$Y_i+F_i^\la +H +\half [Y_i+F_i^\la, H]= Y_i+H+\half[H, Y_i].$$
Thus
$$F_i^\la +\half [F_i^\la, H]+ [Y_i, H]=0.$$
This implies that in the off-center direction the components of $F_i^\la$ are trivial, and for the center direction we have
$$(F_i^\la)_c= -[Y_i, H]$$
In particular this defines coordinate change $H$ which produces conjugate algebraic actions in the family, and each conjugacy class is 2-dimensional determined only by the values  $(F_i^\la)_c$ ($i=1,2$). 

So the 9-dimensional family of algebraic actions modulo the algebraic conjugacy classes gives a 7-dimensional family of non-conjugate algebraic actions. This is the family $\rho^\la$ in Theorem \ref{conj:conjugacydis}.

In the case the manifold is $2n+1$-dimensional Heisenberg nilmanifold, the dimension of the  family of non-conjugate algebraic actions in Theorem \ref{conj:conjugacydis} is  $4n-1$, and the family is described as in \eqref{fam}.
 \color{black}

 \subsection{The commutator operator }\label{commop} Now we analyse the commutator operator for non-algebraic perturbations of translations which  generate $\rho$. Recall that $\rho$ is the $\mathbb Z^2$ action generated by the translation maps $y_i$, $i=1,2$, where $y_i(x)=x\cdot  \exp(Y_i)$ and $Y_i$ are the  two commuting elements in $\frak H$.

 \begin{lemma}\label{comm_operator} Let $F, G\in \vect$ be two sufficiently small vector fields so that the maps  $f(x)= x\cdot \exp(Y_1+F(x))$ and $g(x)= x\cdot \exp(Y_2+G(x))$ are in $\diff$. If $f$ and $g$ commute then the vector fields
 $F$ and $G$ satisfy the following non-linear equation:
  \begin{equation}\label{comm}
F\circ y_2-F + \frac{1}{2}[Y_2, F\circ y_2+F]-(G\circ y_1-G + \frac{1}{2}[Y_1, G\circ y_1+G])+E(F, G)=0
 \end{equation}
where 
 \begin{equation}\label{error_comm}
 \begin{aligned}
E(F, G)&= (F\circ g-F\circ y_2)- (G\circ f-G\circ y_1) \\
&+ \half [Y_2, F\circ g-F\circ y_2]- \half[Y_1, G\circ f-G\circ y_1]\\
&+ \half [G, F\circ g]-\half [F, G\circ f].
 \end{aligned}
 \end{equation}
 \end{lemma}
 \begin{proof}
 Commutation $f\circ g= g\circ f$ implies 
 \[
 x \cdot \exp(Y_2+G)\exp(Y_1+ F\circ g)=x  \cdot\exp(Y_1+F)\exp(Y_2+ G\circ f)\,.
 \]
Hence, 
\[
\begin{aligned}
    & x \cdot \exp(Y_1+ Y_2+ G+ F\circ g+ \half [Y_2, F\circ g]- \half [Y_1, G]+  \half [G, F\circ g])= \\ 
    & x\cdot\exp(Y_1+ Y_2+ F+ G\circ f+ \half [Y_1, G\circ f]- \half [Y_2, F]+  \half [F, G\circ f])\,.
    \end{aligned}
    \]
  The above implies the non-linear equation directly due to the following very simple fact:    $x \cdot \exp(Y+  F)=  x \cdot \exp (Y+G)$ with $Y\in \frak h$ and $F, G\in \vect$, if and only if $F=G$. 
 \end{proof}
 The following immediate consequence of the Lemma above will be used later:
 
 \begin{corollary}\label{c_averages} In the setting of the Lemma \ref{comm_operator},
 $|\Ave [Y_2, F]-\Ave [Y_1, G]|\le C\|F\|_1\|G\|_1$.
 \end{corollary}
 
  \subsection{The conjugation operator}\label{conjop} Here we analyse the conjugation operator for conjugacies close to the identity, we derive the linear part of the conjugacy operator and estimate the error.  
 
 \begin{lemma}\label{conjugation}
 Let $f(x)=x\cdot \exp(Y +F(x))$ be a diffeomorphism of $M$, where $Y\in \frak h$ and $F\in \vect$ is a smooth vector field. Let $h\in \diff$ be a diffeomorphism close to Id given by a small vector field $H\in \vect$ via $h(x)=x\cdot \exp H(x)$. Then $g:=h^{-1}\circ f\circ h$ is a diffeomorphism close  to $f$ given by $G\in \vect$ via    $g(x)=x\cdot \exp(Y +G(x))$ and 
 \[
 G= H-H\circ g+ \half[H+H\circ g, Y] +F\circ h+ [H, F\circ h]+\half [H, H\circ g]-\half [F\circ h, H\circ g]\,.
 \]
  \end{lemma}
  \begin{proof}
  We have 
  \[
 \begin{aligned}
 h^{-1}\circ & f\circ h(x) =x\cdot \exp H(x) \exp(Y+F\circ h(x))\exp(-H\circ h^{-1}\circ f\circ h(x)) \\ 
& =x\cdot \exp(H+Y+F\circ h+\half[H, Y]+[H, F\circ h])(x)\exp(-H\circ g)(x) \\ 
& =x\cdot \exp(H+Y-H\circ g+F\circ h+ \half[H, Y]+ [H, F\circ h] \\ 
 & -\half [H, -H\circ g]-\half[Y, H\circ g]-\half [F\circ h, H\circ g])(x)\,. 
 \end{aligned} 
 \]
  This implies the equality claimed  for $G$. 
  
  
  
  \end{proof}
  

 \subsection{Linearizations of the conjugacy and the commutator operators:  first and second coboundary operators on vector fields, splitting}\label{cohVF}
 
 The linear part of the non-linear equation \eqref{comm} defines the \emph{second coboundary operator} on vector fields over the action $\rho$ generated by $y_1$ and $y_2$: 
 
\begin{definition}
Let $\dd_2: \vect\times\vect \to \vect$ be the linear operator defined by
\[
\dd_2(F, G)= F\circ y_2-F + \frac{1}{2}[Y_2, F\circ y_2+F]-(G\circ y_1-G + \frac{1}{2}[Y_1, G\circ y_1+G])\,. 
\]
 We say that a pair of smooth vector fields $(F, G)$ generates a cocycle over the action $\rho$ if $(F, G)\in Ker \dd_2$.
 \end{definition}
The first coboundary operator on vector fields over the action $\rho$ is given by: 
 \begin{definition}
 Let $H\in \vect$. Then we define $\dd_1:\vect\to \vect\times\vect$ by: 
 \[
 \dd_1(H)=(H\circ y_1-H + \frac{1}{2}[Y_1, H\circ y_1+H], H\circ y_2-H + \frac{1}{2}[Y_2, H\circ y_2+H])\,. 
 \]
 \end{definition}
 It is an easy exercise to check that  
 $Im \dd_1\subset Ker \dd_2$. 
The first cohomology over $\rho$ with coefficients in vector fields is the quotient space $H^1_{\rho}(\vect):= \rm Ker \dd_2/\rm Im \dd_1$. Notice that for constant vector fields $H\in \frak h$ cocycles and coboundaries defined here coincide with those defined in Section~\ref{constantdis}. The subsequent proposition has as a corollary  that for our fixed action $\rho$ the cohomology $H^1_{\rho}(\vect)$ is the same as the cohomology $H^1_{\rho}(\frak h)$ with coefficients in the constant vector fields $\frak h$ which was computed in Section \ref{constantdis}.

 \begin{proposition}\label{splittingVF}
 If the two vector fields $F, G\in \vect$ satisfy $\dd_2(F, G)=\Phi$  and  $(\Ave F, \Ave G) $ 
is in the trivial cohomology class in $H_\rho^1(\frak n)$, then there exists $H, \tilde F, \tilde G\in \vect$ such that $(F, G)= \dd_1H+ (\tilde F, \tilde G)$ and the following estimates hold: 
\begin{equation}\label{estimatesVF}
\begin{aligned}
\|\tilde F\|_s &\leq C_{s} \| \Phi\|_{s+\sigma}\,, \\ 
\|\tilde G\|_s &\leq C_{s}\| \Phi\|_{s+\sigma}\,, \\
\| H \|_s &\leq C_{s, r, S, T, \Gamma}\| F\|_{s+\sigma}
\end{aligned}
\end{equation}

 \end{proposition}
 \begin{proof}
Recall that we disintegrate an  arbitrary vector field $H\in \vect$ into $H= H_c+ H_T$ where $ H_c$ is the  component of $H$ in the direction of $Z$ and $H_T$ is the component of $H$ in all the directions other than $Z$.   So one can view $H_T$ as $H_T=\sum h_t^ iX_i+\bar h_t^i\Lambda_i$, where  $h_t^ i, \bar h_t^i$ are smooth functions. 
 
 The equation $\dd_2(F_T, G_T)=\Phi$ (since  $\dd_2$ is a linear operator) splits then  in the off-center directions into finitely many functional equations each of which has a form $$f\circ y_2-f-(g\circ y_1-g)=\phi$$
 
 Since by assumption $(\Ave F, \Ave G) $  is assumed to be in the trivial constant cohomology class, it implies in particular that all the off-center components are 0 (see Section~\ref{constantdis}). 

Now we may apply Theorem \ref{main:split} which for each of these finitely many equations gives as an output smooth functions $\tilde f, \tilde g, h$ such that $$f=h\circ y_1-f+ \tilde f, \, \, g= g\circ y_2-g+ \tilde g,$$ such that the corresponding estimates hold. Putting these coordinate functions all together gives functions $\tilde F_T, \tilde G_T, H_t$ such that $(\dd_2(\tilde F_T, \tilde G_T))_T=\Phi_T$, $(F_T, G_T)=(\dd_1H_T)_T + (\tilde F_T, \tilde G_T)$ and $\tilde F_T, \tilde G_T, H_T$ satisfy the estimates \eqref{estimatesVF}.

Now let $ F_c,  G_c$ be the components of $F$ and $G$ in the center direction. Then because the $Z$ components within the brackets do not contribute, we have
\[
\begin{aligned} 
\dd_2(F_c, G_c)& = F_c\circ y_2-F_c + \frac{1}{2}[Y_2, F_T\circ y_2+F_T] \\ 
& -(G_c\circ y_1-G_c + \frac{1}{2}[Y_1, G_T\circ y_1+G_T])= \Phi_c\,. 
\end{aligned}
\]
 
 Since we already have $(F_T, G_T)=(\dd_1H_T)_T + (\tilde F_T, \tilde G_T)$, we can substitute this in the above expression to obtain: 
 \begin{equation} \label{xxx}
 \begin{aligned}
 &(F_c-\half[Y_1, H_T\circ y_1+H_T])\circ y_2 -(F_c-\half[Y_1, H_T\circ y_1+H_T])- \\
 &((G_c-\half[Y_2, H_T\circ y_2+H_T])\circ y_1 -(G_c-\half[Y_2, H_T\circ y_2+H_T]))=\Phi'_c\,,
  \end{aligned}
   \end{equation}
where 
\[
\Phi'_c= \Phi_c-\half [Y_2, \tilde F_T\circ y_2+  \tilde F_T ]+ \half [Y_1, \tilde G_T\circ y_1+  \tilde G_T ]\,. 
\]

Clearly since for any $s$ we have $\|\tilde F_T\|_s,\|\tilde G_T\|_s\le \|\tilde \Phi_T\|_{s+\sigma}$, it follows that $\| \Phi'_c\|_s\le \|\tilde \Phi\|_{s+\sigma}$, where $\sigma$ is re-defined to be $\sigma+ 1$. 

The vector field  $H_T$ is determined only up to a constant vector field, so we may choose $H_T$ so that $\Ave F_c= \Ave [Y_1, H_T]$. This forces $F_c-\half[Y_1, H_T\circ y_1+H_T]$ to have the average 0. Moreover, because of the assumption  assumption $(\Ave F, \Ave G) $ is in the trivial cohomology class, we also have that $\Ave F_c= \Ave G_c$. 

 So the equation \eqref{xxx} is again the same type of equation as in Theorem \ref{main:split}. By applying the theorem we get  $\tilde F_c, \tilde G_c, H_c$ such that 
\[ 
\begin{aligned}
& F_c-\half[Y_1, H_T\circ y_1+H_T]= H_c\circ y_1-H_c + \tilde F_c\,, \\ 
& G_c-\half[Y_2, H_T\circ y_2+H_T]= H_c\circ y_2-H_c + \tilde G_c\,. 
\end{aligned}
\]
 This clearly implies 
 \[
 (F_c, G_c)= (\dd_1H)_c + (\tilde F_c, \tilde G_c)\,. 
 \]
 Putting the $c-$ and $T-$ components together gives the solution. Estimates \eqref{estimatesVF} are direct consequence of coordinate-wise estimates which are obtained already in Theorem \ref{main:split}

 \end{proof}

\subsection{Set-up of the perturbative problem and the iterative scheme}\label{conv}

We will frequently refer here to \cite{DK_UNI} so we recommend that the reader has that paper at hand. 

We consider here family of perturbations $\tilde \rho^\la$ of $\rho^\la$, which are generated by commuting maps $\tilde y_1^\la$ and $\tilde y_2^\la$, where for $i=1,2$:
\begin{equation}
\tilde y_i^\la(x)=x\cdot \exp (Y_i+ \tilde F_i^\la), \,\,\,
\end{equation}
Here $\tilde F_i^\la$
 are small vector fields such that $\tilde y_1^\la$ and $\tilde y_2^\la$ commute.

Now let $h$ be a diffeomorphism of the manifold, close to the identity, defined via the smooth vector field $H$ as follows:
 \[
 h(x)= \Gt\cdot \exp H(x)
 \]
 
The iterative step consists of the following: given the perturbation $\tilde \rho^\la$ of $\rho^\la$,  define a new perturbation $\bar  \rho^\la$ which is a conjugation of $\tilde \rho^\la$ via $h$, so $\bar  \rho^\la$ is generated by two diffeomorphisms $\bar y_i^\la$, $i\in \{1,2\}$, defined by 
\[
\bar y_i^\la= h^{-1}\circ \tilde y_i^\la\circ h\,.
\]
In each iterative step this is done for $\lambda$ parameter in some ball, and it is shown that in that ball there is a parameter for which the new family of perturbations is much (quadratically)  closer to $\rho$ for parameters in some smaller ball. The next proposition shows that  that this process is controlled in the sequence of $C^r$ norms. 
 
 We will need to control  derivatives of each perturbed family  $\tilde \rho^\la$ in the direction of the parameter $\lambda$ as well, so we will use the following norms for a family of vector fields $\tilde F_i^\la$:  
$\|\tilde F_i^\la\|_{0, k}$ stands for the supremum of the $C^k$ norms of $\tilde F_i^\la$ in the $\la$ variable.  $\|\tilde F_i^\la\|_{r, k}$ is the same only taken over all the derivatives of  $\tilde F_i^\la$ in the manifold direction. As before, we reserve the notation $\|\tilde F_i^\la\|_{r}$ for the usual $C^r$ norm on $M$ of the vector field $ \tilde F_i^\la\in \vect$ for a fixed parameter $\la$.

The following is an immediate corollary of the classical implicit function theorem and we will use it for the maps which compute averages of vector fields for actions in the perturbed family. 

\begin{lemma}\label{ift} 
There exists an open ball $\mathcal O=\mathcal O(Id, R)$  in $C^1(\mathbb R^d, \mathbb R^d)$, there exists a neighborhood $\mathcal U$ of $0\in \mathbb R^d$ and a $C^1$ map $\Psi: \mathcal O\to \mathcal U$ such that for every $G\in \mathcal O$, $G(\Psi(G))=0$. 
\end{lemma}

Now we state the main iterative step proposition where we show that one can obtain indeed estimates which are needed for the convergence of the process to a smooth conjugation map.

 \begin{proposition}\label{IterativeStep} There exist constants $\bar C$ and $r_0$ such that the following holds:
 
Given  the family $\tilde \rho_n^\la$ of perturbations of $\rho^0$ generated by $\tilde y_{i, n}^\la$ ($i=1,2$), assume that  for all $\la$ in a ball $B$ centered at 0, for $r\in \mathbb N$ and $t>0$:
 
 1)$\| \tilde F_{i, n}^\la\|_0\le \e_n<1$, 
 
 2) The map $\la\mapsto \tilde F_{i, n}^\la$ is $C^2$ in $\la$, $\|\tilde F_{i, n}^\la\|_{0, 1}\le \e_n$ and $\|\tilde F_{i, n}^\la\|_{0, 2}\le K_n$, 
 
 3) The map $\Phi^n : \la\mapsto Ave(\tilde F_{i, n}^\la)$ is in $\mathcal O$ and has a zero at $\la_n$,
 
 4) $\| \tilde F_{i, n}^\la\|_{r_0+r}\le \delta_{r, n}$, 
 
 5) $t^{r_0}\e _n^{1-\frac{1}{r+r_0}}\delta_{r}^{\frac{1}{r+r_0}}<\bar C$.
 
 There exists a $H_n\in \vect$ such that $h_n$ defined by 
 $
 h(x)= x\cdot \exp H_n(x)
 $ such that the newly formed family of perturbations $\tilde \rho_{n+1}^\la$ of $\rho$, generated by $\tilde y_{i, n+1}^\la= h^{-1}\circ \tilde y_{i, n}^\la\circ h$, with $\tilde y_{i, n+1}^\la(x)= x \exp (Y_i+ \tilde F_{i, n+1}^\la(x))$, satisfies the following:
 
 a) $\|H_n\|_r\le C_rt^{2r_0}\|\tilde F_i^{\la_n}\|_r$.
 
 b) $ \|\tilde F_{i, n+1}^{\la}\|_0\le K_n\|\la-\la_n\|+ \rm Err(t, r)$ where 
 
 \begin{equation*}
\begin{aligned}{\rm Err}_{n+1}(t,r)&:=C \e_n^2
 +C\delta_{r, n}^{\frac{r_0+1}{r_0+r}}\e_n^{2-\frac{r_0+1}{r_0+r}}+C_{r}t^{-r}\delta_{r,n}\\&+C_{r_0}t^{2r_0}\e_n^{2-\frac{1}{r_0+r}}\delta_{r,n}^{\frac{1}{r_0+r}}+C_{r_0}t^{2r_0}\e_n^{3-\frac{1}{r_0+r}}\delta_{r,n}^{\frac{2}{r_0+r}}.
 \end{aligned}
 \end{equation*}

 
 c)$\| \tilde  F_{i, n+1}^\la\|_{r_0+r}\le C_rt^{2r_0}\delta _{r, n}:=\delta _{r, n+1} $.
  
 d) The map $\Phi^{n+1}:=\la\mapsto \Ave(\tilde F_{i, n+1}^\la)$ satisfies 
\begin{align*}
\|\Phi^{n+1}-\Phi^n\|_{(0)}&\le {\rm Err}_{n+1}(t,r)\\
\|\Phi^{n+1}-\Phi^n\|_{(1)}&\le K_nt^{r_0}\e_n+{\rm Err}_{n+1}(t,r).
\end{align*}

If $\Phi^{n+1}$ is in $\mathcal O$, then it has a zero at $\la_{n+1}\in B$ which satisfies $$\|\la_{n+1}-\la_{n}\|\le C{\rm Err}_{n+1}(t,r)+CK_n( K_nt^{r_0}\e_n+{\rm Err}_{n+1}(t,r))^2.$$

(e) $\tilde F_{i, n+1}^\la$ is $C^2$ in $\la$ and $$\| \tilde F_{i, n+1}^\la\|_{0, 2}\le (1+Ct^{r_0}\e_n^{1-\frac{1}{r+r_0}}\delta_{r,n}^{\frac{1}{r+r_0}})K_n=:K_{n+1}(t,r).$$

   \end{proposition}
   
   \begin{proof}
 
 As was mentioned in \cite[Remark 6.3]{DK_UNI} the proof of the iterative step is universal given tame splitting for vector fields (Proposition \ref{splittingVF}). We repeat the main points here for the sake of completeness with few less details then in the proof of the corresponding proposition in \cite[Proposition 6.2]{DK_UNI}.

  In this proof, as is customary whenever there is a loss of regularity for solutions of linearized equations, we will use the smoothing operators. For the construction of smoothing operators on $C^\infty(M)$ see \cite{Hamilton}: Example 1.1.2. (2), Definition 1.3.2, Theorem 1.3.6, Corollary 1.4.2. 
There exists a collection of smoothing operators $S_t:C^\infty(M)\to C^\infty(M)$, $t>0$, such that the following holds:

\begin{equation}\label{smoothing}
\begin{aligned}
\|S_tF\|_{s+s'}&\le C_{s,s'}t^{s'}\|F\|_s\\
\|(I-S_t)F\|_{s-s'}&\le C_{s,s'}t^{-s'}\|F\|_s
\end{aligned}
\end{equation}

Smoothing operators on  $C^\infty(M)$ clearly induce smoothing operators on $\vect$ via smoothing operators applied to coordinate maps.

It is easy to see that averages of $F$ with respect to the Haar measure on $M$, in various directions in the tangent space do not affect the properties of smoothing operators listed above, so without loss of generality we may assume that $S_t$ are such that 
averages of $S_tF$ are the same as those of $F$.

Given $\tF$ we first apply the smoothing operators to it and write $\tF= S_t\tF+ (I-S_t) \tF$. Now $\Ave(\tF)= \Ave (S_t\tF)$. From the commutativity of $\tF$ for $i=1$ and $i=2$ (see Corollary \ref{c_averages}) it follows that  $|\Ave [Y_2, \tilde F_{1, n}]-\Ave [Y_1, \tilde F_{2, n}]|\le C\|\tilde F_{1, n}\|_1\|\tilde F_{2, n}\|_1\le C\e_n^2$ and clearly the same holds after application of the corresponding smoothing operators. Now we can apply Proposition \ref{splittingVF} to $S_t\tF- \Ave (\tF)_T$, $i=1, 2$ (recall that $\Ave (\tF)_T$ are averages in the off-center direction). Proposition \ref{splittingVF} gives existence of $H_n$ such that 
\[
\|(S_t\tilde F_{1, n}- \Ave (\tilde F_{1, n})_T, S_t\tilde F_{2, n}- \Ave (\tilde F_{2, n})_T ) - \d_1 H_n\| _r\le C\|\Phi\|_{r+\sigma}\,,
\] 
where (see  \eqref{error_comm})
\[
\Phi:= E(\tilde F_{1, n}, \tilde F_{2, n})\,. 
\]
From the expression for $E$ in \eqref{error_comm} we have the following estimate for $\Phi$:
\[
\|\Phi\|_r\le C \|\tF\|_r\|\tF\|_{r+1}\,, 
\]
where we use short notation  $ \|\tF\|_r$ for the maximum of the norms for $i=1$ and $i=2$. 
Also from Proposition \ref{splittingVF} we have 
\[
\| H \|_r \leq C\|  S_t\tilde F_{i, n}^\la-\Ave (\tilde F_{1, n})_T\|_{r+\sigma}\le Ct^\sigma \|\tilde F_{i, n}^\la\|_r\,. 
\]
From the Lemma \ref{conjugation} it follows that if we define $h_n$ by $h_n(x) = x\cdot \exp H_n(x)$, and $\tilde y_{i, n+1}^\la= h^{-1}\circ \tilde y_{i, n}^\la\circ h$, with $\tilde y_{i, n+1}^\la(x)= x\cdot  \exp (Y_i+ \tilde F_{i, n+1}^\la(x))$, then $ \tilde F_{i, n+1}^\la$ satisfy the following, after applying the interpolation estimates and the smoothing estimates and assumptions 2) and 3) (compare to (6.7) in \cite{DK_UNI}): 
\begin{equation}
\begin{aligned}
\|\tilde F_{i, n+1}^\la\|_0&\le  K_n\|\la-\la_n\| +  C \e_n^2+C\delta_{r, n}^{\frac{r_0+1}{r_0+r}}\e_n^{2-\frac{r_0+1}{r_0+r}}\\
&+C_{r}t^{-r}\delta_{r,n}+C_{r_0}t^{2r_0}\e_n^{2-\frac{1}{r_0+r}}\delta_{r,n}^{\frac{1}{r_0+r}}+C_{r_0}t^{2r_0}\e_n^{3-\frac{1}{r_0+r}}\delta_{r,n}^{\frac{2}{r_0+r}}
\end{aligned}
\end{equation} 
For the $C^{r_0+r}$ norm of the new error $\tilde F_{i, n+1}^\la$,  as usual in this type of proofs,  we only need a "linear" bound with respect to the corresponding norm of the old error. This follows easily from the conjugacy relation and we obtain  for any $s\ge 0$ :
\[
\|\tilde F_{i, n+1}^\la\|_s \le  C_st^{2r_0}\|\tilde F_{i, n}^\la\|_s\,, 
\]
which as in \cite{DK_UNI} implies 
\[
\|\tilde F_{i, n+1}^\la\|_s \le  C_st^{2r_0}\delta_{r, n}\,.
\]
Remaining two statements (e) and (d) follow exactly in the same way as in proof of \cite[Proposition 6.2]{DK_UNI} 
 \end{proof}

Given the proposition \ref{IterativeStep} (compare to \cite[Proposition 6.2]{DK_UNI}) we can now apply the  convergence of the successive iterative scheme proved  in \cite[Section 7]{DK_UNI}. Consequently we obtain the following Theorem, which is a more precise statement of our main transversal local rigidity result in Theorem \ref{conj:conjugacydis}:

\begin{theorem} There exist $l>0$, $\epsilon >0$, $R>0$, such that 
if a family $\tilde \rho^\la$ of perturbations of $\rho$ generated by $\tilde y_i^\la$ is $\epsilon$ close to $\rho$ in the  $C^l$ norm for parameters $\lambda$ in an $R$- ball around 0, and in the $C^1$ norm in the parameter $\lambda$ direction, 
then there exists a small parameter $\bar \lambda$ such that  the action $\tilde \rho^{\bar \la}$ is conjugate to $\rho$ via $h$, that is  for $i=1, 2$ we have:
\[
h\circ y_i= \tilde y_i^{\bar \la}\circ h\,,
\]
 where $h$ is a smooth diffeomorphism order of $\epsilon$ close to the identity in the $C^1$ norm. 
 \end{theorem}

\appendix
\section{Proof of Propositions~\ref{prop:finite-reps-cocycle} and \ref{prop:finite-reps}}


The classical Diophantine condition \eqref{eq:diophantine_property} stated in Section~\ref{setting} is clearly equivalent to the following condition: 
there are constants $c := c_{\boldsymbol{\tau}, \boldsymbol{\eta}} > 0$ and $\gamma := \gamma_{\boldsymbol{\tau}, \boldsymbol{\eta}} > 0$ such that  
for any $\mathbf m \in \Z^{2n}$ and $p \in \Z$,  we have 
\begin{equation}\label{eq:diophantine_property2}
\begin{aligned}
& \vert \boldsymbol{\tau} \cdot \mathbf m - p \vert > c \vert \mathbf m \cdot \mathbf m \vert^{-\gamma} & \text{ if } \mathbf m_1 \neq  0\,, \\ 
&  \vert \boldsymbol{\eta} \cdot \mathbf m - p \vert > c \vert \mathbf m \cdot \mathbf m \vert^{-\gamma} & \text{ if } \mathbf m_2 \neq  0\,.
\end{aligned}
\end{equation}
We will use the above version of the Diophantine condition to prove the splitting results for finite dimensional representations in this section. The same splitting results were needed and used in three other works so far: \cite{DF}, \cite{WX} and \cite{P}, and they follow closely  Moser's splitting construction on the circle in \cite{Moser}. Our presentation here is somewhat different in that it follows a general splitting construction which applies to abelian actions where cohomological equations in irreducible representations have finite dimensional space of obstructions (as in \cite{DT}, for example).
 \smallskip

For any $\mathbf m \in \Z^{2n}$ and for any 
$\boldsymbol \kappa \in \{\boldsymbol \tau, \boldsymbol \eta\}$, 
define the constant $\zeta(\mathbf m, \boldsymbol\kappa)$ by 
\[
\zeta(\mathbf m, \boldsymbol\kappa) :=  \exp(2\pi i (\mathbf m \cdot \boldsymbol{\kappa})) - 1\,.
\]
The next lemma describes the operator $L_{\boldsymbol{\kappa}}$ 
on smooth functions in $L^2(\T^{2n})$.  Its proof is straightforward and follows from the diophantine condition \eqref{eq:diophantine_property2}.  
\begin{lemma}\label{lemm:mathcalL:3}
Let $h = \sum_{\mathbf m \in \Z^{2n}} h_{\mathbf m} \exp(2\pi i \mathbf m \cdot (\mathbf x, \boldsymbol \xi))\in W_0^\infty(\T^{2n})$ be a smooth, zero average function with coefficients ($h_{\mathbf m})$.    
Then for $\boldsymbol\kappa \in \{\boldsymbol\tau, \boldsymbol\eta\}$, 
\[
L_{\boldsymbol{\kappa}} h(\mathbf x, \boldsymbol \xi) = \sum_{\mathbf m \in \Z^{2n}\setminus\{0\}} h_{\mathbf m} \zeta(\mathbf m, \boldsymbol\kappa) \exp(2\pi i \mathbf m \cdot (\mathbf x, \boldsymbol \xi))\,.  
\]
Moreover, there is a constant $C_{\boldsymbol\tau, \boldsymbol\eta} > 0$ 
such that for any $\mathbf m \in \Z^{2n} \setminus\{0\}$, 
\[
\begin{aligned}
&  \vert\zeta(\mathbf m, \boldsymbol\tau)\vert^{-1} \leq C_{\boldsymbol{\tau}, \boldsymbol{\eta}}\vert \mathbf m \cdot \mathbf m\vert^{\gamma} & \text{ if } \mathbf m_1 \neq 0\,, \\ 
&  \vert\zeta(\mathbf m, \boldsymbol\eta)\vert^{-1} \leq C_{\boldsymbol{\tau}, \boldsymbol{\eta}} \vert \mathbf m \cdot \mathbf m\vert^{\gamma} & \text{ otherwise}\,,
\end{aligned}
\]
where for $j = 1, 2$, $\mathbf m_j$ 
is defined in \eqref{eq:m1,m2}, and $\gamma$ is the exponent in the diophantine condition for $\boldsymbol{\tau}, \boldsymbol{\eta}$, see \eqref{eq:diophantine_property2}.  
\end{lemma}
 
We now prove Theorem~\ref{main:transfer} 
in the context of finite dimensional representations. 
\begin{proof}[Proof of Proposition~\ref{prop:finite-reps-cocycle}]
Because $f, g \in W^\infty(\T^{2n})$ are zero average functions, 
there are coefficients $(g_{\mathbf m}) \,, (f_{\mathbf m})  \in \ell^2(\Z^{2n})$ 
with $f_{\mathbf 0} = g_{\mathbf 0} = 0$\,,
such that 
\begin{equation}\label{fg-def-finite}
\begin{aligned}
g = \sum_{\mathbf m \in \Z^{2n}\setminus\{0\}} g_{\mathbf m} \exp(2\pi i \mathbf m \cdot (\mathbf x, \boldsymbol \xi))\,, \\ 
f = \sum_{\mathbf m \in \Z^{2n}\setminus\{0\}} f_{\mathbf m} \exp(2\pi i \mathbf m \cdot (\mathbf x, \boldsymbol \xi))\,. 
\end{aligned}
\end{equation}
By Lemma~\ref{lemm:mathcalL:3}, we get  
\[
\begin{aligned}
& L_{\boldsymbol{\tau}} g(\mathbf x, \boldsymbol \xi) = \sum_{\mathbf m \in \Z^{2n}\setminus\{0\}} g_{\mathbf m} \zeta(\mathbf m, \boldsymbol\tau) \exp(2\pi i \mathbf m \cdot (\mathbf x, \boldsymbol \xi))\,, \\ 
& L_{\boldsymbol{\eta}} f(\mathbf x, \boldsymbol \xi) = \sum_{\mathbf m \in \Z^{2n}\setminus\{0\}} f_{\mathbf m} \zeta(\mathbf m, \boldsymbol\eta) \exp(2\pi i \mathbf m \cdot (\mathbf x, \boldsymbol \xi))\,.  
\end{aligned}
\]

Then because ($\exp(2\pi i \mathbf m\cdot (\mathbf x, \boldsymbol \xi)))_{\mathbf m \in \Z^{2n}}$ is an orthogonal basis for $L^2(\T^{2n})$, 
$L_{\boldsymbol{\tau}} g = L_{\boldsymbol{\eta}} f$ implies 
that for any $\mathbf{m} \in \Z^{2n}\setminus\{0\}$, 
\[
g_{\mathbf m} \zeta(\mathbf m, \boldsymbol\tau) = f_{\mathbf m} \zeta(\mathbf m, \boldsymbol\eta)\,.
\]
From the definition of $\zeta$ and the diophantine 
property for $\boldsymbol \tau$ and $\boldsymbol \eta$ (see \eqref{eq:diophantine_property2}), 
we get that for any $\mathbf m \in \Z^{2n} \setminus\{0\}$, 
\begin{equation}\label{eq:zeta:4}
\left\{
\begin{aligned}
& \zeta(\mathbf m, \boldsymbol\tau) \neq 0 & \text{ if } \mathbf m_1 \neq 0\,, \\ 
& \zeta(\mathbf m, \boldsymbol\tau) = 0 & \text{ if } \mathbf m_1 = 0\,, \\ 
& \zeta(\mathbf m, \boldsymbol\eta) \neq 0 & \text{ if } \mathbf m_1 = 0\,. 
\end{aligned}
\right. 
\end{equation} 
Hence, 
\begin{equation}\label{eq:gn-fnC}
\begin{aligned}
& g_{\mathbf m} = f_{\mathbf m} \frac{\zeta(\mathbf m, \boldsymbol\eta)}{\zeta(\mathbf m, \boldsymbol\tau)} \,, & \text{ if } \mathbf m_1 \neq 0  \,,  \\ 
& f_{\mathbf{n}} = 0\,, &  \text{otherwise} \,.
\end{aligned}
\end{equation}

Now define the sequence ($P_{\mathbf m})_{\mathbf{m} \in \Z^{2n}}$ by $P_{\mathbf{0}} = 0$ and for any nonzero $\mathbf{m}$, set 
\begin{equation}\label{def:P}
P_{\mathbf m} :=  \left\{
\begin{aligned}
& \frac{f_{\mathbf m}}{\zeta(\mathbf m, \boldsymbol\tau)}\,, & \text{ if } \mathbf m_1 \neq 0\,,  \\ 
& \frac{g_{\mathbf{n}}}{\zeta(\mathbf m, \boldsymbol\eta)} \,,  & \text{ otherwise}\,.   
\end{aligned}
\right.  
\end{equation}
Let  
\[
P := \sum_{\mathbf m \in \Z^{2n}} P_{\mathbf m} \exp(2\pi i \mathbf{n} \cdot(\mathbf x, \mathbf{\xi}))\,.
\]
Then a calculation formally gives $
L_{\boldsymbol{\tau}} P = f\,, \quad L_{\boldsymbol{\eta}} P = g\,, $
where the first equation follows from the second equalities in \eqref{eq:zeta:4} and \eqref{eq:gn-fnC}, and the second equation follows from the first equality in \eqref{eq:gn-fnC} and equation \eqref{def:P}.

Now we estimate the Sobolev norm of $P$.  
Recall from \eqref{eq:norm_sob_finite_dim} that for any $f \in W^\infty(\T^{2n})$ and for any $s \in \N$, 
\[
\Vert f \Vert_s = \Vert (1 + 4\pi^2 (\mathbf m \cdot \mathbf m))^{s/2} f \Vert_{\ell^2(\Z^{2n})} < \infty\,. 
\]
Set $s \in \N$.  
By Lemma~\ref{lemm:mathcalL:3} and formula \eqref{def:P}, for any 
$\mathbf m \in \Z^{2n}\setminus\{0\}$ such that $\mathbf m_1 \neq 0$, 
we have 
\[
\begin{aligned}
\vert P_{\mathbf m}\vert 
= \frac{\vert f_{\mathbf m}\vert}{\vert \zeta(\mathbf m, \boldsymbol\tau)\vert} 
\leq C_{\boldsymbol \tau, \boldsymbol \eta} (1 + 4\pi^2 \mathbf m \cdot \mathbf m)^\gamma \vert f_{\mathbf m}\vert\,.
\end{aligned}
\]
On the other hand, when $\mathbf m_1 = 0$ we have 
\[
\begin{aligned}
\vert P_{\mathbf m}\vert 
= \frac{\vert g_{\mathbf m}\vert}{\vert \zeta(\mathbf m, \boldsymbol\eta)\vert} 
\leq C_{\boldsymbol \tau, \boldsymbol \eta} (1 + 4\pi^2 \mathbf m \cdot \mathbf m)^\gamma \vert g_{\mathbf m}\vert\,.
\end{aligned}
\]

Then for any $s \in \N$, 
when there is a constant $C_{\boldsymbol{\tau}, \boldsymbol{\eta}} > 0$ 
such that 
\begin{align}
\Vert P \Vert_{s}^2 & = \sum_{\substack{\mathbf{m} \in \Z^{2n}\setminus\{0\} \\ 
\mathbf m_1 \neq 0}}
(1 + 4\pi^2 \mathbf m \cdot \mathbf m)^{s} \vert P_{\mathbf m}\vert^2  
 +  \sum_{\substack{\mathbf{m} \in \Z^{2n}\setminus\{0\} \\ 
\mathbf m_1 = 0}} (1 + 4\pi^2 \mathbf m \cdot \mathbf m)^{s} \vert P_{\mathbf m}\vert^2 \notag \\ 
& 
\leq C_{\boldsymbol \tau, \boldsymbol \eta} \sum_{\mathbf{m} \in \Z^{2n}\setminus\{0\}}  (1 + 4\pi^2 \mathbf m \cdot \mathbf m)^{s+2\gamma} (\vert f_{\mathbf m}\vert^2 + \vert g_{\mathbf m}\vert^2)\,. \label{eq:P-est:1_0}
\end{align}
By interpolation, the above estimate holds for any $s \geq 0$.  
Hence, for any $s \geq 0$, 
\[
\begin{aligned}
\eqref{eq:P-est:1_0} & = C_{\boldsymbol{\tau}, \boldsymbol{\eta}} ( \Vert f \Vert_{s + 2\gamma}^2 +  \Vert g \Vert_{s + 2\gamma}^2) 
\leq C_{\boldsymbol{\tau}, \boldsymbol{\eta}} ( \Vert f \Vert_{s + 2\gamma}  + \Vert g \Vert_{s + 2\gamma})^2  \,.
\end{aligned}
\]
We conclude 
\[
\Vert P \Vert_{s} \leq C_{\boldsymbol{\tau}, \boldsymbol{\eta}} ( \Vert f \Vert_{s + 2\gamma} +    \Vert g \Vert_{s + 2\gamma})\,.
\]
\end{proof}

Furthermore, we have a tame splitting in first cohomology with coefficients in smooth functions, which establishes Theorem~\ref{main:split} for the case of finite dimensional representations.  
\begin{proof}[Proof of Proposition~\ref{prop:finite-reps}]
Let $s \in \N$.  Let $f, g$ be given by \eqref{fg-def-finite}, and write $\phi$ as 
\[
\phi = \sum_{\mathbf m \in \Z^{2n}} \phi_{\mathbf m} \exp(2\pi i \mathbf m \cdot (\mathbf x, \boldsymbol \xi))\,, 
\]
where $(\phi_{\mathbf m}) \in \ell^2(\Z^{2n})$.  
Because $\phi = L_{\boldsymbol{\eta}} f - L_{\boldsymbol{\tau}} g$, we get 
\[
\phi_{\mathbf 0} = 0\,.
\]
By assumption, $f$ and $g$ also have zero average, so 
\[
f_{\mathbf 0} = g_{\mathbf 0} = 0\,,  
\]
Define $P$ by the sequence $(P_{\mathbf m})_{\mathbf m \in \Z}$ given in \eqref{def:P}, where $P_{\mathbf 0} = 0$. 

Let $R$ be orthogonal projection in $L^2(\T^{2n})$ onto the space generated by 
\[
\bigcup_{\substack{\mathbf m \in \Z^{2n} \\ \mathbf m_1 \neq 0 }} \{\exp(2\pi i \mathbf m \cdot (\mathbf x, \boldsymbol \xi))\}\,.
\]  
That is, for any $h =  \sum_{\mathbf m \in \Z^{2n}} h_{\mathbf m} \exp(2\pi i (\mathbf m \cdot (\mathbf x, \boldsymbol \xi)))$ 
in $L^2(\T^{2n})$, 
\begin{equation}\label{def:R-finite}
\begin{aligned}
R h(\mathbf x, \boldsymbol \xi) 
& = \sum_{\substack{\mathbf m \in \Z^{2n} \\ 
\mathbf m_1 \neq 0}} h_{\mathbf m} \exp(2\pi i (\mathbf m \cdot (\mathbf x, \boldsymbol \xi))) \,.
\end{aligned}
\end{equation}

A direct calculation gives the next lemma.  
\begin{lemma}\label{lemm:RLtau-eta}
The following equalities hold on $L^2(\T^{2n})$,  
\[
\begin{aligned} 
 R L_{\boldsymbol{\eta}} = L_{\boldsymbol{\eta}} R\,, \quad  R L_{\boldsymbol{\tau}} = L_{\boldsymbol{\tau}}\,.
\end{aligned}
\]
\end{lemma}
Now let $P$ be defined by \eqref{def:P}.   
Then 
\[
\begin{aligned}
L_{\boldsymbol{\tau}} P & = \sum_{\substack{\mathbf m \in \Z^{2n} \\ 
\mathbf m_1 \neq 0}} f_{\mathbf m} \exp(2\pi i \mathbf m \cdot (\mathbf x, \boldsymbol \xi)) 
= R f\,.
\end{aligned}
\] 
By the above equality and Lemma~\eqref{lemm:RLtau-eta}, we get as in \eqref{eq:R_phi}: $R\phi 
 = L_{\boldsymbol{\tau}} (L_{\boldsymbol{\eta}} P - g)\,.$
From \eqref{def:R-finite}, it follows that for any $\mathbf m \in \Z^{2n}$ such that $\mathbf m_1 = 0$, $(R \phi)_{\mathbf m} = 0\,.$
Moreover, for any $h \in L^2(\T^{2n}),$ we get from the definition of $L_{\boldsymbol{\tau}}$ that for such $\mathbf m$, 
\[
(L_{\boldsymbol{\tau}} h)_{\mathbf{m}} = 0\,.
\]
This means 
\[
\begin{aligned}
\sum_{\substack{\mathbf m \in \Z^{2n}\\ \mathbf m_1 \neq 0}} & (R \phi)_{\mathbf m} \exp(2\pi i \mathbf m \cdot (\mathbf x, \boldsymbol \xi)) 
=  R \phi(\mathbf x, \boldsymbol \xi) 
 = L_{\boldsymbol{\tau}} (L_{\boldsymbol{\eta}} P - g) (\mathbf x, \boldsymbol \xi) \\ &  
 = \sum_{\substack{\mathbf m \in \Z^{2n} \\ \mathbf m_1 \neq 0}} (L_{\boldsymbol{\eta}}P - g)_{\mathbf m} \zeta(\mathbf m, \boldsymbol\tau)\exp(2\pi i \mathbf m \cdot (\mathbf x, \boldsymbol \xi)) \,.
\end{aligned}
\]

By orthogonality, it follows that for all $\mathbf{m} \in \Z^{2n} \setminus \{0\}$ 
with $\mathbf m_1 \neq 0$, 
\begin{equation}\label{eq:rphi}
(R \phi)_{\mathbf m} = (L_{\boldsymbol{\eta}}P - g)_{\mathbf m} \zeta(\mathbf m, \boldsymbol\tau)\,.
\end{equation}
Note that the definition of $P_{\mathbf m}$ gives 
\[
\begin{aligned}
(L_{\boldsymbol{\eta}} P - g)_{\mathbf m} & = P_{\mathbf m} \zeta(\mathbf m, \boldsymbol\eta) - g_{\mathbf m} 
= 0\,. 
\end{aligned}
\]
So by the above equality, formula~\eqref{eq:rphi} and Lemma~\ref{lemm:mathcalL:3}, 
we get that for any $\mathbf m \in \Z^{2n}$,  
\[
\begin{aligned}
\vert (L_{\boldsymbol{\eta}}P - g)_{\mathbf m}\vert 
 = \frac{\vert (R \phi)_{\mathbf m}\vert}{\vert \zeta(\mathbf m, \boldsymbol\tau)\vert} 
 \leq C_{\boldsymbol \tau, \boldsymbol \eta} (1 + 4\pi^2 \mathbf m \cdot \mathbf m)^\gamma \vert (R \phi)_{\mathbf m}\vert \,.
\end{aligned}
\]

Hence, 
\[
\begin{aligned}
\Vert L_{\boldsymbol{\eta}} P - g\Vert_s^2 & = \sum_{\substack{\mathbf m \in \Z^{2n} \\ \mathbf m_1 \neq 0}} (1 + 4\pi^2 \mathbf m \cdot \mathbf m)^s \vert (L_{\boldsymbol{\eta}} P - g)_\mathbf m \vert^2 \\ 
& \leq C_{\boldsymbol \tau, \boldsymbol \eta} \sum_{\substack{\mathbf m \in \Z^{2n} \\ \mathbf m_1 \neq 0}} (1 + 4\pi^2 \mathbf m \cdot \mathbf m)^{s + 2\gamma} \vert (R \phi)_{\mathbf m} \vert^2  \\ 
& = C_{\boldsymbol \tau, \boldsymbol \eta} \Vert R \phi \Vert_{s + 2\gamma}^2 
\leq C_{\boldsymbol \tau, \boldsymbol \eta} \Vert \phi \Vert_{s + 2\gamma}^2\,.
\end{aligned} 
\]

Next, as in \eqref{eq:L1P-f}, we get  
\[
\begin{aligned}
\Vert L_{\boldsymbol{\tau}} P - f \Vert_s 
= \Vert (R - I) f \Vert_s \,.
\end{aligned}
\]
By Lemma~\ref{lemm:RLtau-eta}, it follows as in \eqref{eq:(R_I)L_eta} that 
\[
L_{\boldsymbol{\eta}} (R - I) f = (R - I) L_{\boldsymbol{\eta}} f 
= (R - I) \phi\,.
\]
Next, a calculation proves 
that for any $\mathbf m \in \Z^{2n} \setminus\{0\}$ such that $\mathbf m_1 = 0$, 
\[
\begin{aligned}
\vert ((R - I)f)_\mathbf m\vert & \leq  \frac{\vert ((R - I)\phi)_\mathbf m \vert }{\vert \zeta(\mathbf m, \boldsymbol\eta) \vert}  
\leq C_{\boldsymbol \tau, \boldsymbol \eta} (1 + 4\pi^2 \mathbf m \cdot \mathbf m)^\gamma \vert ((R - I)\phi)_\mathbf m \vert \,.
\end{aligned}
\]
Then using Lemma~\ref{lemm:mathcalL:3} we conclude that 
\[
\begin{aligned}
\Vert (R - I)f \Vert_s & \leq C_{\boldsymbol \tau, \boldsymbol \eta} \Vert (R - I)\phi \Vert_{s + 2\gamma}  
\leq C_{\boldsymbol \tau, \boldsymbol \eta} \Vert \phi \Vert_{s + 2\gamma}\,.
\end{aligned}
\]

The third inequality in Proposition~\ref{prop:finite-reps} holds  
because $P$ is the same function from 
Proposition~\ref{prop:finite-reps-cocycle}, 
which gives 
\[
\begin{aligned}
\Vert P \Vert_s & \leq C_{\boldsymbol{\tau}, \boldsymbol{\eta}} (\Vert f \Vert_{s +2\gamma} + \Vert g \Vert_{s + 2\gamma})\,.
\end{aligned}
\]

Now if $P$ is nonconstant, then we are done.  
So suppose that $P$ is constant, and therefore zero.  
Notice that by the above estimate, $\phi = 0$ implies that $f = 0$, which 
contradicts the assumption that $f \neq 0$.  
So we conclude that there is some $\mathbf m_0 \in \Z^{2n}$ such that 
\[
\phi_{\mathbf m_0} \neq 0\,.
\]
Then define 
\[
\widetilde P(\mathbf x, \boldsymbol \xi) := \phi_{\mathbf m_0}\exp(2\pi i \mathbf m_0 \cdot (\mathbf x, \boldsymbol \xi))\,.
\]  
By the orthogonal decomposition of $\phi$, we have $
\Vert \widetilde P \Vert_s \leq \Vert \phi \Vert_s\,.$
So the above estimates of $\Vert L_{\boldsymbol{\eta}} P - g\Vert_s$ and $\Vert L_{\boldsymbol{\tau}} P - f \Vert_s$ imply 
\[
\begin{aligned}
\|L_{\boldsymbol{\eta}} \widetilde P-g\|_s & = \Vert (L_{\boldsymbol{\eta}} P-g) + L_{\boldsymbol{\eta}} \widetilde P\Vert_s \\ 
& \leq \Vert L_{\boldsymbol{\eta}} P-g\Vert_s + \Vert L_{\boldsymbol{\eta}} \widetilde P\Vert_s \\ 
& \leq (C_{\boldsymbol \tau, \boldsymbol \eta} +1)\|\phi \|_{s+2\gamma}\,,
\end{aligned}
\] 
and analogously, 
$\Vert L_{\boldsymbol{\tau}} \widetilde P-f \Vert_s \leq (C_{\boldsymbol \tau, \boldsymbol \eta} +1)\|\phi \|_{s+2\gamma}\,.$

This concludes the proof of Proposition~\ref{prop:finite-reps}.  
\end{proof}

\end{document}